\documentclass[a4paper, 10pt, parskip=half]{scrartcl}

\usepackage[utf8]{inputenc}
\usepackage[T1]{fontenc}
\usepackage{lmodern}
\usepackage{csquotes}
\usepackage{amsmath}
\usepackage{amssymb}
\usepackage{amsthm}
\usepackage{amsfonts}
\usepackage{dsfont}
\usepackage{mathtools}
\usepackage{marginnote}
\usepackage[dvipsnames]{xcolor}
\usepackage{hyperref}
\hypersetup{%
  colorlinks=true,
  linkcolor=black,
  citecolor=black,
  urlcolor=blue,
  pdftitle={Chemotaxis(-fluid) systems with logarithmic sensitivity and slow consumption: global generalized solutions and eventual smoothness},
  pdfauthor={Mario Fuest},
  pdfkeywords={},
  bookmarksopen=true,
}

\usepackage{xcolor}

\usepackage[numbers, sort&compress]{natbib}
\newcommand{\R}{\mathbb{R}}

\newcommand{\N}{\mathbb{N}}

\newcommand{\mc}[1]{\mathcal{#1}}

\newcommand{\ur}[1]{\mathrm{#1}}
\newcommand{\ure}{\ur{e}}

\ifdefined\labelenumi%
  \renewcommand{\labelenumi}{(\roman{enumi})}
  
\fi

\newcommand{\eps}{\varepsilon}

\DeclareMathOperator{\supp}{supp}

\newcommand{\defs}{\coloneqq}
\newcommand{\sfed}{\eqqcolon}

\newcommand{\ra}{\rightarrow}

\newcommand{\sea}{\searrow}

\newcommand{\rh}{\rightharpoonup}

\newcommand{\ol}{\overline}

\newcommand{\dx}{\,\mathrm{d}x}
\newcommand{\ds}{\,\mathrm{d}s}

\newcommand{\dsigma}{\,\mathrm{d}\sigma}

\newcommand{\ddt}{\frac{\mathrm{d}}{\mathrm{d}t}}

\newcommand{\embed}{\hookrightarrow}

\newcommand{\hp}{\hphantom}
\newcommand{\pe}{\mathrel{\hp{=}}}

\newcommand{\intom}{\int_\Omega}

\newcommand{\intntom}{\int_0^T \int_\Omega}

\newcommand{\intninfom}{\int_0^\infty \int_\Omega}

\newcommand{\Ombar}{\ol \Omega}

\newcommand{\Ombarinf}{\Ombar \times [0, \infty)}

\newcommand{\loc}{\mathrm{loc}}

\newcommand{\leb}[2][\Omega]{\ensuremath{L^{#2}(#1)}}
\newcommand{\lebl}[1][\Omega]{\ensuremath{L\log L(#1)}}
\newcommand{\sob}[3][\Omega]{\ensuremath{W^{#2, #3}(#1)}}

\newcommand{\con}[2][\Ombar]{\ensuremath{C^{#2}(#1)}}

\newcommand{\dual}[1]{\ensuremath{(#1)^\star}}

\newcommand{\neps}{n_\eps}
\newcommand{\net}{n_{\eps t}}
\newcommand{\nne}{n_{0 \eps}}
\newcommand{\ce}{c_\eps}
\newcommand{\cet}{c_{\eps t}}
\newcommand{\cne}{c_{0 \eps}}
\newcommand{\we}{w_\eps}
\newcommand{\wet}{w_{\eps t}}
\newcommand{\wne}{w_{0 \eps}}
\newcommand{\ue}{u_\eps}
\newcommand{\uet}{u_{\eps t}}
\newcommand{\une}{u_{0 \eps}}

\newcommand{\cej}{c_{\eps_j}}

\newcommand{\infc}{\delta}
\newcommand{\tops}{\texorpdfstring}

\makeatletter
\renewenvironment{proof}[1][\proofname]{\par
  \pushQED{\qed}%
  \normalfont \topsep0\p@\relax
  \trivlist
  \item[\hskip\labelsep\scshape
  #1\@addpunct{.}]\ignorespaces
}{%
  \popQED\endtrivlist\@endpefalse
}
\makeatother

\newtheorem{base}{Base}[section]
\numberwithin{equation}{section}

\newtheorem{theorem}[base]{Theorem} \newtheorem*{theorem*}{Theorem}
\newtheorem{lemma}[base]{Lemma} \newtheorem*{lemma*}{Lemma}
 \newtheorem*{prop*}{Proposition}
 \newtheorem*{cor*}{Corollary}

\theoremstyle{definition}
\newtheorem{remark}[base]{Remark} \newtheorem*{remark*}{Remark}
\newtheorem{definition}[base]{Definition} \newtheorem*{definition*}{Definition}
 \newtheorem*{example*}{Example}
 \newtheorem*{cond*}{Condition}

\newif\ifhyperconst             

\makeatletter
\newcounter{globalconst}
\newcommand{\newgc}[2][]{%
\refstepcounter{globalconst}%
\ltx@label{gc:#2}%
\ifhyperconst
  \hyperref[gc:#2]{C}_{\ref{gc:#2}#1}
\else
  C_{\begin{NoHyper}\ref{gc:#2}\end{NoHyper}}
\fi}
\ifhyperconst
  \newcommand{\gc}[2][]{\hyperref[gc:#2]{C}_{\ref{gc:#2}#1}}
\else
  \newcommand{\gc}[2][]{C_{\begin{NoHyper}\ref{gc:#2}\end{NoHyper}}}
\fi

\newcounter{localconst}[base]
\newcommand{\newlc}[2][]{%
\refstepcounter{localconst}%
\ltx@label{lc:\thesection:\arabic{base}:#2}%
\ifhyperconst
  \hyperref[lc:\thesection:\arabic{base}:#2]{c}_{\ref{lc:\thesection:\arabic{base}:#2}#1}
\else
  c_{\begin{NoHyper}\ref{lc:\thesection:\arabic{base}:#2}\end{NoHyper}}
\fi}
\ifhyperconst
  \newcommand{\lc}[2][]{\hyperref[lc:\thesection:\arabic{base}:#2]{c}_{\ref{lc:\thesection:\arabic{base}:#2}#1}}
\else
  \newcommand{\lc}[2][]{c_{\begin{NoHyper}\ref{lc:\thesection:\arabic{base}:#2}\end{NoHyper}}}
\fi
\makeatother

\setkomafont{title}{\normalfont\Large}
\title{Chemotaxis(-fluid) systems with logarithmic sensitivity and slow consumption: global generalized solutions and eventual smoothness}
\usepackage{authblk}

\author{Mario Fuest\footnote{e-mail: fuest@ifam.uni-hannover.de, corresponding author}}
\affil{Leibniz Universität Hannover, \\ Institut für Angewandte Mathematik, \\ Welfengarten 1, 30167 Hannover, Germany}
\date{}

\begin{document}
\maketitle

\KOMAoptions{abstract=true}
\begin{abstract}
\noindent
We consider the system
\begin{align*}
  \begin{cases}
    n_t + u \cdot \nabla n = \Delta n - \chi \nabla \cdot (\frac{n}{c} \nabla c), \\
    c_t + u \cdot \nabla c = \Delta c - nf(c), \\
    u_t + (u \cdot \nabla) u = \Delta u + \nabla P + n \nabla \phi, \quad \nabla \cdot u = 0,
  \end{cases}
\end{align*}
in smooth bounded domains $\Omega \subset \mathbb R^N$, $N \in \mathbb N$, for given $f \ge 0$, $\phi$ and complemented with initial and homogeneous Neumann--Neumann--Dirichlet boundary conditions,
which models aerobic bacteria in a fluid drop.
We assume $f(0) = 0$ and $f'(0) = 0$, that is, that $f$ decays slower than linearly near $0$,
and construct global generalized solutions provided that either $N=2$ or $N > 2$ and no fluid is present.\\[0.5pt]
If additionally $N=2$,
we next prove that this solution eventually becomes smooth and stabilizes in the large-time limit.
We emphasize that these results require smallness neither of $\chi$ nor of the initial data.\\[0.5pt]
 \textbf{AMS Classification (2020):} {35K55 (primary); 35B40, 35B65, 35D99, 35Q35, 92C17 (secondary)}
\end{abstract}
\section{Introduction}
This article is concerned with the chemotaxis--Navier--Stokes system with logarithmic sensitivity 
\begin{align}\label{prob:fluid_incl}
  \begin{cases}
    n_t + u \cdot \nabla n = \Delta n - \chi \nabla \cdot (\frac{n}{c} \nabla c)             & \text{in $\Omega \times (0, \infty)$}, \\
    c_t + u \cdot \nabla c = \Delta c - nf(c)                                                & \text{in $\Omega \times (0, \infty)$}, \\
    u_t + (u \cdot \nabla) u = \Delta u + \nabla P + n \nabla \phi, \quad \nabla \cdot u = 0 & \text{in $\Omega \times (0, \infty)$}, \\
    \partial_\nu n = \partial_\nu c = 0, \, u= 0                                             & \text{on $\partial \Omega \times (0, \infty)$}, \\
    (n, c, u)(\cdot, 0) = (n_0, c_0, u_0)                                                    & \text{in $\Omega$},
  \end{cases}
\end{align}
variants of which have been proposed in \cite{TuvalEtAlBacterialSwimmingOxygen2005}
to model the behavior of aerobic bacteria interacting with a fluid by means of transportation and buoyancy.
Apart from random motion (term $\Delta n$) the bacteria with density $n$ may also partially orient their movement
towards higher concentrations of oxygen with density $c$ (term $-\chi \nabla \cdot (\frac{n}{c} \nabla c)$, where $\chi > 0$ is given), a mechanism called chemotaxis.
The oxygen diffuses (term $\Delta c$) and is consumed by the bacteria (term $-nf(c)$, where $f$ is a given nonnegative function)
and both the bacteria and the oxygen are transported by a fluid with velocity field $u$ (terms $+u \cdot \nabla n$ and $+u \cdot \nabla c$),
which in turn is given by the Navier--Stokes equation with an inhomogeneity accounting for buoyancy effects,
i.e.\ the observation that water with bacteria is heavier than without (term $+n \nabla \phi$, where $\phi$ is a given gravitational potential).  
The modelling in \cite{TuvalEtAlBacterialSwimmingOxygen2005} is motivated by experiments performed in \cite{DombrowskiEtAlSelfConcentrationLargeScaleCoherence2004};
for an alternative derivation of such systems based on a multiscale approach see \cite{BellomoEtAlMultiscaleDerivationNonlinear2016}.


The mathematical analysis of such models has first focussed on variants with linear taxis sensitivity, i.e.\ with $-\chi \nabla \cdot (\frac nc \nabla c)$ in the first equation replaced by $-\chi \nabla \cdot (n \nabla c)$,
for which global classical \cite{CaoLankeitGlobalClassicalSmalldata2016} and weak solutions \cite{DuanEtAlGlobalSolutionsCoupled2010,WinklerGlobalLargedataSolutions2012} have been constructed for small and large data, respectively.
Beyond that, many more results on chemotaxis--fluid systems (for instance also covering situations where the signal is produced rather than consumed) are available;
we refer to \cite[Section~4.6]{BellomoEtAlChemotaxisCrossdiffusionModels2022} for a recent overview.

The reason behind considering a logarithmic sensitivity in \eqref{prob:fluid_incl} is to account for the so-called Weber--Fechner law of stimulus perception,
an idea already present in earlier models by Keller and Segel \cite{KellerSegelTravelingBandsChemotactic1971}
(see also \cite{RosenSteadystateDistributionBacteria1978} and \cite[Section~2.2]{HillenPainterUserGuidePDE2009}).

From a mathematical point of view, the logarithmic sensitivity poses two main challenges:
First, it destroys the (quasi-) energy structure crucially made use of for instance in \cite{DuanEtAlGlobalSolutionsCoupled2010} and \cite{WinklerGlobalLargedataSolutions2012} inter alia for linear sensitivity functions
as $c f(c)$ cannot be concave near $c=0$ for nonnegative smooth $f$ with $f(0) = 0$ (cf.\ the discussion in \cite[p.~1638]{DuanEtAlGlobalSolutionsCoupled2010}).
Second, the consumption term in the second equation in \eqref{prob:fluid_incl} may force $c$ to become very small,
which means that the factor $\frac1c$ in the taxis term may become very large.
In fact, even local-in-time positive lower bounds for $c$ appear to be unavailable without relying on (generally not available) upper bounds for $n$---%
in contrast to logarithmic chemotaxis systems with signal production where even time-independent lower bounds have been derived, see for instance \cite[Lemma~2.2]{FujieSenbaGlobalExistenceBoundedness2015}.

Even without fluid interaction, global classical solutions of \eqref{prob:fluid_incl} have only been constructed under smallness conditions \cite{WangEtAlAsymptoticDynamicsSingular2016}
or for related, more regular systems
(in \cite{LankeitLocallyBoundedGlobal2017} with nonlinear diffusion enhancement,
in \cite{LiuGlobalClassicalSolution2018} with saturated taxis sensitivity,
and in \cite{LankeitViglialoroGlobalExistenceBoundedness2020} with weaker consumption terms, i.e.\ for $-n f(c)$ replaced by $-n^\beta c$ for some $\beta \in (0, 1)$).

\paragraph{Main result I: Global generalized solutions.}
The lack of unconditional global existence results for classical solutions is not entirely surprising.
Systems including cross-diffusive terms such as $-\chi \nabla \cdot (\frac{n}{c} \nabla c)$ in \eqref{prob:fluid_incl} generally have quite low regularity properties.
Indeed, the probably most prominent representative of such models, the classical Keller--Segel system
\begin{align*}
  \begin{cases}
    n_t = \Delta n - \nabla \cdot (n \nabla c), \\
    c_t = \Delta c - c + n
  \end{cases}
\end{align*}
admits classical solutions blowing up in finite time in both two- \cite{HerreroVelazquezBlowupMechanismChemotaxis1997, MizoguchiWinklerBlowupTwodimensionalParabolic}
and higher-dimensional \cite{WinklerFinitetimeBlowupHigherdimensional2013} settings%
---as do many variants of this system,
see for instance \cite{NagaiSenbaGlobalExistenceBlowup1998} for blow-up in a parabolic--elliptic system with logarithmic sensitivity
and the recent survey \cite{LankeitWinklerFacingLowRegularity2019} for further examples.
This makes it necessary to consider weaker solution concepts when studying global solvability of cross-diffusive systems
and while sometimes switching to usual weak formulations suffices, one often needs to consider even more general concepts;
see \cite[Section~1.2]{FuestStrongConvergenceWeighted2022} for a more thorough discussion of various notions of solvability.

Such generalized solutions have also been obtained for \eqref{prob:fluid_incl} with $f(c) = c$
both in the two-dimensional setting with fluid
(\cite{LiuLargetimeBehaviorTwodimensional2021};
see also the precedents \cite{WangGlobalLargedataGeneralized2016} for the chemotaxis--Stokes system, i.e.\ \eqref{prob:fluid_incl} without the nonlinear convection term $(u \cdot \nabla) u$ in the fluid equation,
and \cite{WinklerTwodimensionalKellerSegel2016} for the fluid-free case)
and higher dimensional fluid-free radially symmetric settings (\cite{WinklerRenormalizedRadialLargedata2018}).

Our first main result expands on these findings and states that \eqref{prob:fluid_incl} possesses global generalized solutions if either $N=2$ or if $N \ge 3$ and there is no fluid present,
and if $f$ does not only fulfill
\begin{align}\label{eq:intro:cond_f}
  f \in C^1([0, \infty)) 
  \quad \text{with} \quad
  f(0) = 0
  \quad \text{and} \quad
  f > 0 \text{ in $(0, \infty)$}
\end{align}
but also
\begin{align}\label{eq:ev_smooth:cond_f}
  f'(0) = 0.
\end{align}
By \eqref{eq:intro:cond_f}, the condition \eqref{eq:ev_smooth:cond_f} is equivalent to $\lim_{s \sea 0} \frac{f(s)}{s} = 0$;
that is, \eqref{eq:ev_smooth:cond_f} requires superlinear decay of $f$ near $0$, which reflects that oxygen is consumed \emph{more slowly} if nearly none is left.
We remark that \eqref{eq:ev_smooth:cond_f} is not fulfilled for the most typical choice $f(c) = c$ but, for instance, for $f(c) = c^2$.

\begin{theorem}\label{th:global_ex}
  Let $\Omega \subset \R^N$, $N \ge 2$, be a smooth, bounded domain, $\chi > 0$ and $\phi \in \sob2\infty$,
  and assume \eqref{eq:intro:cond_f}, \eqref{eq:ev_smooth:cond_f}
  and
  \begin{align}\label{eq:intro:cond_init}
    \begin{cases}
      n_0 \in \leb1 \quad \text{with $n_0 > 0$ a.e.\ and $\ln n_0 \in \leb1$}, \\
      c_0 \in \leb\infty \quad \text{with $c_0 \ge \infc$ for some $\infc > 0$}, \\
      u_0 \in L^2(\Omega; \R^N) \quad \text{with $\nabla \cdot u_0 = 0$ in $\mc D'(\Omega)$}.
    \end{cases}
  \end{align}
  If
  \begin{align}\label{eq:intro:n=2_or_no_fluid}
    N = 2
    \quad \text{or} \quad
    (u_0 \equiv 0 \text{ and } \phi \equiv 0),
  \end{align}
  then there exists a global generalized solution $(n, c, u)$ of \eqref{prob:fluid_incl} in the sense of Definition~\ref{def:sol_concept} below
  which has the property that $u \equiv 0$ if $u_0 \equiv 0$ and $\phi \equiv 0$.
\end{theorem}

The main novelty of Theorem~\ref{th:global_ex} is the fluid-free higher dimensional case.
Indeed, for $\chi = 1$, more regular initial data and $f(c) = c$ the two-dimensional setting with fluid has already been treated in \cite{LiuLargetimeBehaviorTwodimensional2021}
and the challenges caused by large $\chi$, less regular initial data and more general $f$ are relatively limited.
(In fact, the main reason to include the two-dimensional existence result here is so that we can refer to it in Theorem~\ref{th:ev_smooth} below.)

On the other hand, new ideas are necessary for $N \ge 3$.
While quasi-energy functionals quickly give a~priori bounds for (weighted) gradients of all components of solutions to approximative systems (see Lemma~\ref{lm:quasi_energy} and Lemma~\ref{lm:fluid_energy}),
they are only strong enough to imply uniform integrability of $\neps f(\ce)$ in the two-dimensional setting (where such an estimate follows for instance from Heihoff's inequality \cite[(1.2)]{HeihoffTwoNewFunctional2022}).
However, such an a~priori estimate turns out to be crucial for undertaking the limit processes in the weak formulations for the first and second solution components;
see the discussion at the beginning of Subsection~\ref{sec:strong_conv} for a more thorough discussion of this point.
One way to obtain uniform integrability of $\neps$ (and hence also of $\neps f(\ce)$) is to modify the system by adding a superlinear degradation term to the first equation in \eqref{prob:fluid_incl}.
This observation stands as the core of \cite[Section~7]{FuestStrongConvergenceWeighted2022}, where recently global generalized solution for the corresponding problem have been constructed
(cf.\ also the precedent \cite{LankeitLankeitGlobalGeneralizedSolvability2019} for quadratic dampening terms).

However, as we aim to prove a global existence result for \eqref{prob:fluid_incl} without any dampening terms and also for $N \ge 3$,
these ideas alone are evidently insufficient for our purpose.
Fortunately, we can adapt an idea recently introduced in \cite{HombergEtAlExistenceGeneralizedSolutions2022}:
The crucial a~priori estimate, uniform boundedness of $\intntom \neps f(\ce) |\ln(\neps f(\ce))|$,
follows eventually from considering the time evolution of the functional $\intom (\ln \neps) \ce$, see Lemma~\ref{lm:nc_uniform_int} and the discussion directly preceding that lemma.

Additionally relying on \eqref{eq:ev_smooth:cond_f}, we are then also able to favourably bound $\nabla \ln \ce$ on sets where $\ce$ is small (cf.\ Lemma~\ref{lm:nabla_ln_c_eta}),
which allows us to obtain strong convergence of $\nabla \ln \ce$ in Lemma~\ref{lm:strong_conv2}.
Combined, these bounds and convergence properties then make it possible to pass to the limit in each term of the weak formulations for the approximate problems,
allowing us to prove Theorem~\ref{th:global_ex} in Subsection~\ref{sec:pf_th1}.

\paragraph{Main result II: Eventual smoothness and stabilization in the two-dimensional setting.}
A natural next step is to analyze the behavior of the solution given by Theorem~\ref{th:global_ex} for large times
with respect to both eventual regularization and convergence in the large-time limit.
These two points are actually related:
Both bounds in comparatively strong topologies and smallness of certain quantities may be key in favourably estimating worrisome terms in energy functionals,
which, depending on the functional, in turn may imply stronger a~priori estimates or convergence (see also the discussion after Theorem~\ref{th:ev_smooth} below).

Indeed, for certain relatives of \eqref{prob:fluid_incl} with scalar nonsingular taxis sensitivity,
both convergence to homogeneous steady states 
(see for instance \cite{WinklerStabilizationTwodimensionalChemotaxisNavier2014, JiangEtAlGlobalExistenceAsymptotic2015} for the two-dimensional
as well as \cite{ChaeEtAlGlobalExistenceTemporal2014, CaoLankeitGlobalClassicalSmalldata2016} for the small-data three-dimensional setting
and \cite{WinklerHowFarChemotaxisdriven2017} for results regarding weak solutions)
and eventual smoothness properties
(see e.g.\ \cite{TaoWinklerEventualSmoothnessStabilization2012} for the fluid-free three-dimensional case
and \cite{WangGlobalSolvabilityEventual2020} for a three-dimensional chemotaxis--Navier--Stokes system with superlinear degradation)
have been shown.
Moreover, for tensor-valued taxis sensitivities, global generalized solutions have been constructed in planar domains
(\cite{HeihoffGlobalMasspreservingSolutions2020}; see also \cite{WinklerLargedataGlobalGeneralized2015} for a precedent dealing with a Stokes fluid),
whose large-time and eventual regularity properties have been analysed in \cite{HeihoffTwoNewFunctional2022}.

Regarding chemotaxis-fluid systems with logarithmic sensitivity, the results appear to be limited to two-dimensional settings:
For Stokes fluids, \cite{WangGlobalLargedataGeneralized2016} shows convergence towards homogeneous steady states
while \cite{BlackEventualSmoothnessGeneralized2018} asserts eventual smoothness provided $\intom n_0$ is sufficiently small.
Moreover, under a similar smallness condition, both convergence and eventual smoothness are obtained in \cite{LiuLargetimeBehaviorTwodimensional2021} for the full Navier--Stokes equation.

Our second main result is then able to give an affirmative answer to the question whether similar relaxation properties are also exhibited by the global generalized solution (potentially with large mass)
constructed in Theorem~\ref{th:global_ex}.
\begin{theorem}\label{th:ev_smooth}
  Suppose in addition to the assumption of Theorem~\ref{th:global_ex} that $N=2$.
  Then the solution given by Theorem~\ref{th:global_ex} eventually becomes smooth in the sense that there are $t_\star > 0$ and $P \in C^{1, 0}(\Ombar \times [t_\star, \infty))$ such that
  \begin{align*}
    n, c \in C^{2, 1}(\Ombar \times [t_\star, \infty)), \quad 
    u \in C^{2, 1}(\Ombar \times [t_\star, \infty); \R^2)
  \end{align*}
  and that $(n, c, u, P)$ is a classical solution of \eqref{prob:fluid_incl} in $\Ombar \times [t_\star, \infty)$;
  that is, that the first four equations therein are fulfilled pointwise in $\Ombar \times [t_\star, \infty)$.

  Moreover, this solution stabilises in the large time limit. More precisely,
  \begin{align}\label{eq:ev_smooth:large_time}
    \lim_{t \to \infty} \left( \|n(\cdot, t) - \ol n_0\|_{\con2} + \|c(\cdot, t)\|_{\con2} + \|u(\cdot, t)\|_{\con2} \right) = 0,
  \end{align}
  where $\ol n_0 = \frac{1}{|\Omega|} \intom n_0$.
\end{theorem}
This theorem relates to \cite[Theorem~1.2]{LiuLargetimeBehaviorTwodimensional2021} which proves eventual smoothness by requiring smallness of $\intom n_0$ instead of \eqref{eq:ev_smooth:cond_f}.
A key ingredient to both results is the energy functional \eqref{eq:def_f}, which is conditional in the sense that it only decreases throughout evolution if certain conditions are met,
namely that the functional is already sufficiently small at some time $t_0 \ge 0$.
That this functional indeed dissipates under certain conditions is verified in Lemma~\ref{lm:cond_func}, where we already rely on \eqref{eq:ev_smooth:cond_f}.

In order to actually make use of this conditional energy structure we then need to show that various quantities become small at some point in time.
The starting point is the quasi-energy inequality \eqref{eq:quasi_energy:ddt} holding for approximate solutions $(\neps, \ce, \ue)$, $\eps \in (0, 1)$.
Its right-hand side $\chi^2 \intom \neps \ce^{-1} f(\ce)$ becomes small if $f(\ce) = \ce$ and $\intom \neps = \intom n_0$ is small;
this is a core idea of \cite[Lemma~5.2]{LiuLargetimeBehaviorTwodimensional2021}.
Not wanting to impose a smallness condition of the initial data, we are forced to argue differently.
A key observation is that $\neps \ce^{-1} f(\ce)$ can be controlled favourably for $\ce$ bounded away from $0$
while for small $\ce$ the key condition \eqref{eq:ev_smooth:cond_f} shows that the respective term becomes small as well.
Thus, splitting the right-hand side of \eqref{eq:quasi_energy:ddt} in integrals over suitable subdomains allows us to obtain smallness of the first part of the conditional energy functional,
see Lemma~\ref{lm:nabla_ln_n_nabla_w_uniform_spacetime} and Lemma~\ref{lm:u_space_time}.

For the remaining term, we crucially rely on the inequalities \eqref{eq:heihoff_ineq:1} and \eqref{eq:heihoff_ineq:2} recently derived by Heihoff in \cite{HeihoffTwoNewFunctional2022}.
They improve on related inequalities by getting rid of an additional additive $L^1$ term on the right-hand side,
which make them favourably applicable also for large $\intom n_0$, see Lemma~\ref{lm:c_l1_small}.

Finally, the bounds implied by the conditional energy functional serve as a starting point for a bootstrap procedure sketched in Lemma~\ref{lm:c2_bdd},
which provides further estimates sufficiently strong to arrive at Theorem~\ref{th:ev_smooth} in Subsection~\ref{sec:proof_ev_smooth}.

\paragraph{Main result III: Global existence of classical solutions under smallness conditions.}
As a byproduct of the arguments developed for proving Theorem~\ref{th:ev_smooth}, 
we immediately obtain global classical solutions emanating from initial data already satisfying the conditions for the conditional energy functional.
\begin{theorem}\label{th:global_ex_classical}
  Let $\Omega \subset \R^2$ be a smooth, bounded domain, $\chi > 0$ and $\phi \in \sob2\infty$,
  suppose that $f$ complies with \eqref{eq:intro:cond_f} and \eqref{eq:ev_smooth:cond_f}
  and let $m_0 > 0$.
  There exists $\eta > 0$ such that whenever
  \begin{align*}
    (n_0, c_0, u_0) \in \con0 \times \sob1q \times \mc D(A^\beta)
    \qquad \text{for some $q > 2$ and $\beta \in (\tfrac12, 1)$},
  \end{align*}
  where $A$ denotes the Stokes operator,
  are such that $n_0 > 0$ and $c_0 > 0$ in $\Ombar$,
  \begin{align*}
    \intom n_0 = m_0
    \quad
    \|c_0\|_{\leb\infty} \le m_0
  \end{align*}
  and
  \begin{align*}
    \intom n_0 \ln \frac{n_0}{\ol n_0} + \intom c_0^2 + \intom \frac{|\nabla c_0|^2}{c_0^2} + \intom |u_0|^2 \le \eta,
  \end{align*}
  then there exists a global classical solution $(n, c, u, P)$ of \eqref{prob:fluid_incl}, which moreover satisfies \eqref{eq:ev_smooth:large_time}.
\end{theorem}

\paragraph{Plan of the paper.}
We first collect some useful general inequalities in Section~\ref{sec:prelims}
before proving Theorem~\ref{th:global_ex} in Section~\ref{sec:global_ex} and Theorem~\ref{th:ev_smooth} as well as Theorem~\ref{th:global_ex_classical} in Section~\ref{sec:ev_smooth};
we refer to the beginning of the later two sections for a discussion of the finer structure.
Let us here just point directly to the most crucial new steps:
Lemma~\ref{lm:nc_uniform_int} asserts uniform integrability of $\neps f(\ce)$ 
and key ideas relying on \eqref{eq:ev_smooth:cond_f} are employed in Lemma~\ref{lm:nabla_ln_c_eta}, Lemma~\ref{lm:cond_func} and Lemma~\ref{lm:nabla_ln_n_nabla_w_uniform_spacetime}.

\paragraph{Notation.}
Let $\Omega \subset \R^N$, $N \in \N$, be a smooth, bounded domain.
Throughout the article, we abbreviate $\ol \varphi \defs \frac1{|\Omega|} \intom \varphi$ for $\varphi \in \leb1$,
set $L_\sigma^2(\Omega; \R^N) \defs \{\,\varphi \in \leb2 : \nabla \cdot \varphi = 0 \text{ in } \mc D'(\Omega)\,\}$, where $\mc D'(\Omega)$ is the space of distributions on $\Omega$,
as well as $W_{0, \sigma}^{1, 2}(\Omega; \R^N) \defs W_0^{1, 2}(\Omega; \R^N) \cap L_{\sigma}^2(\Omega; \R^N)$.
We further denote by $\mc P$ the Helmholtz projection on $L^2(\Omega; \R^N)$ and by $A$ the Stokes operator $-\mc P \Delta$ on $L^2(\Omega; \R^N)$ with domain $W^{2, 2}(\Omega; \R^N) \cap W_{0, \sigma}^{1, 2}(\Omega; \R^N)$,
and often write $\leb2$ instead of $L^2(\Omega; \R^N)$ etc.\ when the codomain can be inferred from the context.
Moreover, uppercase constants $C_i$ are unique throughout the article while lowercase constants $c_i$ are ``local'' to each proof.

\section{Preliminaries: The Csisz\'ar--Kullback and Heihoff inequalities}\label{sec:prelims}
This preliminary section is concerned with multiple estimates regarding $\intom \varphi \ln \frac{\varphi}{\ol \varphi}$ and $\intom \frac{|\nabla \varphi|^2}{\varphi^2}$ for positive functions $\varphi$.
First, we recall the classical Csisz\'ar--Kullback inequality which gives a nontrivial lower bound for the former term.
\begin{lemma}\label{lm:csiszar-kullback}
  Let $\Omega \subset \R^N$, $N \in \N$, be a smooth, bounded domain.
  For all nonnegative $0 \not\equiv \varphi \in \lebl$, the inequality
  \begin{align}\label{eq:csiszara-kullback:statement}
    \left( \intom |\varphi - \ol \varphi| \right)^2 \le 2 \left(\intom \varphi \right) \left( \intom \varphi \ln \frac{\varphi}{\ol \varphi} \right)
  \end{align}
  holds. 
\end{lemma}
\begin{proof}
  We first suppose that $0 \le \varphi \in \lebl$ is such that $\ol \varphi = 1$.
  Then we can apply \cite[Theorem~4.1]{CsiszarInformationtypeMeasuresDifference1967} 
  to the probability measures $\mu_1(\mathrm dx) \defs \varphi(x) \mc L(\mathrm dx)$ and $\mu_2(\mathrm dx) = \mc L(\mathrm dx)$,
  where $\mc L$ denotes the Lebesgue measure on $\Omega$ multiplied with $\frac1{|\Omega|}$,
  to obtain
  \begin{align}\label{eq:csiszara-kullback:est_int_1}
    \left( \frac1{|\Omega|} \intom |\varphi - 1| \right)^2 \le \frac2{|\Omega|} \intom \varphi \ln \varphi.
  \end{align}
  For arbitrary nonnegative $0 \not\equiv \varphi \in \lebl$,
  we infer \eqref{eq:csiszara-kullback:statement} from \eqref{eq:csiszara-kullback:est_int_1} applied to $\frac{\varphi}{\ol \varphi}$.
\end{proof}

Next, we state two inequalities recently derived by Heihoff in \cite{HeihoffTwoNewFunctional2022},
which improve on the inequalities obtained in \cite[Lemma~3.2]{HeihoffGlobalMasspreservingSolutions2020} (see also \cite[Lemma~2.2 and Lemma~2.3]{WinklerSmallmassSolutionsTwodimensional2020})
as consequences of the Trudinger–Moser inequality (cf.\ \cite{TrudingerImbeddingsOrliczSpaces1967} and \cite{MoserSharpFormInequality1970}).
These improvements entail getting rid of an additive term of the form $\intom \psi$ at the cost of potentially enlarging certain multiplicative constants
and, as already noticed in the introduction, form an essential cornerstone in the proof of Theorem~\ref{th:ev_smooth}.
\begin{lemma}\label{lm:heihoff_ineq}
  Let $\Omega \subset \R^2$ be a smooth, bounded domain. Then there exists $\newgc{heihoff} > 0$ such that
  \begin{align}\label{eq:heihoff_ineq:1}
        \intom \varphi (\psi - \ol \psi)
    \le \frac{\eta}{4 \gc{heihoff}} \left( \intom \psi \right) \intom |\nabla \varphi|^2 + \frac1\eta \intom \psi \ln \left(\frac{\psi}{\ol \psi} \right)
  \end{align}
  holds for all $\varphi \in \con1$ and $0 < \psi \in \con0$ and all $\eta > 0$,
  and that
  \begin{align}\label{eq:heihoff_ineq:2}
        \intom \psi \ln \left( \frac{\psi}{\ol \psi} \right)
    \le \frac{1}{\gc{heihoff}} \left( \intom \psi \right) \intom \frac{|\nabla\psi|^2}{\psi^2}
  \end{align}
  holds for all $0 < \psi \in \con1$.
\end{lemma}
\begin{proof}
  See \cite[Theorem~1.1]{HeihoffTwoNewFunctional2022}.
  While that theorem actually additionally requires that $\Omega$ is finitely connected, this assumption turns out to be superfluous.
  Indeed, as $\partial \Omega$ is smooth and compact, it is the union of finitely many connected compact one-dimensional manifolds,
  all of which are diffeomorphic to circles by the classification theorem \cite[theorem in the appendix]{MilnorTopologyDifferentiableViewpoint1965}.
  Therefore, $\R^2 \setminus \Omega$ contains only finitely many connected components.
\end{proof}

\section{Global existence of generalized solutions}\label{sec:global_ex}
In this section, we prove Theorem~\ref{th:global_ex}, i.e.\ the global existence of generalized solutions to \eqref{prob:fluid_incl}.
To that end, we fix throughout this section a smooth, bounded domain $\Omega \subset \R^N$, $N \ge 2$, $\chi > 0$, $\phi \in \sob2\infty$ as well as $f$ and $(n_0, c_0, u_0)$
complying with \eqref{eq:intro:cond_f}--\eqref{eq:intro:cond_init}.

We start by introducing and discussing our notion of global generalized solutions in Subsection~\ref{sec:sol_concept}
before we construct global classical solutions to approximate problems in Subsection~\ref{sec:ge_approx}, for which $\eps$-independent bounds are obtained in Subsection~\ref{sec:apriori}.
In Subsection~\ref{sec:strong_conv}, we then come to the main novelty compared to the predecessors \cite{WangGlobalLargedataGeneralized2016} and \cite{LiuLargetimeBehaviorTwodimensional2021}:
Lemma~\ref{lm:nc_uniform_int} shows uniform integrability of $\neps f(\ce)$ also in higher spatial dimensions,
which allows us to obtain convergence properties sufficient for taking the limit in each term of the weak formulations.

\subsection{Solution concept}\label{sec:sol_concept}
Our notion of generalized solvability follows well-established definitions in the context of chemotaxis systems with low regularity;
see for instance \cite[Definition~2.3]{WinklerLargedataGlobalGeneralized2015} for an early example,
\cite{FuestStrongConvergenceWeighted2022} for an overview
and \cite[Definition~4.1]{LiuLargetimeBehaviorTwodimensional2021} for a closely related setting.
That is, while for $c$ and $u$ usual weak formulations apply,
for the first solution component we only require that $\ln(n+1)$ is a weak supersolution of the corresponding equation and that its mass is bounded by the initial mass.
\begin{definition}\label{def:sol_concept}
  We call a triple
  \begin{align*}
    (n, c, u) \in L^\infty((0, \infty); \leb1) \times L^\infty(\Omega \times (0, \infty)) \times L_{\loc}^2([0, \infty); W_0^{1,2}(\Omega; \R^N))
  \end{align*}
  with $n \ge 0, c > 0, \nabla \cdot u = 0$ a.e.\ in $\Omega \times (0, \infty)$ and
  \begin{align*}
    \ln(n+1), \, \ln c \in L_{\loc}^2([0, \infty); \sob12)
  \end{align*}
  a \emph{global generalized solution} of \eqref{prob:fluid_incl} if
  \begin{itemize}
    \item
      $n$ is a weak $\ln$-supersolution of the first subproblem in \eqref{prob:fluid_incl} in the sense that
      \begin{align}\label{eq:sol_concept:n_ln_supersol}
        &\pe  - \intninfom \ln(n+1) \varphi_t
              - \intom \ln(n_0+1) \varphi(\cdot, 0) \notag \\
        &\ge  \intninfom \frac{|\nabla n|^2}{(n+1)^2} \varphi
              - \intninfom \frac{\nabla n \cdot \nabla \varphi}{n+1} \notag \\
        &\pe  - \chi \intninfom \frac{n \nabla n \cdot \nabla c}{(n+1)^2c} \varphi
              + \chi \intninfom \frac{n \nabla c \cdot \nabla \varphi}{(n+1)c} \notag \\
        &\pe  + \intninfom \ln(n+1) (u \cdot \nabla \varphi)
        \qquad \text{for all $0 \le \varphi \in C_c^\infty(\Ombarinf)$},
      \end{align}

    \item
      there is a null set $N \subset (0, \infty)$ such that
      \begin{align}\label{eq:sol_concept:mass_ineq}
        \intom n(\cdot, t) \le \intom n_0
        \qquad \text{for all $t \in (0, \infty) \setminus N$},
      \end{align}

    \item
      $c$ is a weak solution of the second subproblem in \eqref{prob:fluid_incl} in the sense that
      \begin{align}\label{eq:sol_concept:v_weak_sol}
          - \intninfom c \varphi_t
          - \intom c_0 \varphi(\cdot, 0)
        = - \intninfom \nabla c \cdot \nabla \varphi
          - \intninfom n f(c) \varphi
          + \intninfom c (u \cdot \nabla \varphi)
      \end{align}
      for all $\varphi \in C_c^\infty(\Ombarinf)$ and

    \item
      $u$ is a weak solution of the third subproblem in \eqref{prob:fluid_incl} in the sense that
      \begin{align}\label{eq:sol_concept:u_weak_sol}
          - \intninfom u \cdot \varphi_t
          - \intom u_0 \cdot \varphi(\cdot, 0)
        = - \intninfom \nabla u \cdot \nabla \varphi 
          + \intninfom (u \otimes u) \cdot \nabla \varphi
          + \intninfom n \nabla \phi \cdot \varphi
      \end{align}
      for all $\varphi \in C_c^\infty(\Omega \times [0, \infty); \R^N)$ with $\nabla \cdot \varphi = 0$ in $\Omega \times (0, \infty)$.
  \end{itemize}
\end{definition}

\begin{remark}\label{rm:sol_concept}
  \begin{enumerate}
    \item 
      Let us note that this solution concept is compatible with the notion of classical solutions
      in the sense that if $(n, c, u)$ is a \emph{smooth} global generalized solution in the sense of Definition~\ref{def:sol_concept},
      then $(n, c, u)$ together with some $P \in C^{1, 0}(\Ombar \times (0, \infty))$ is also a classical solution.
      Indeed, for $c$ and $u$ this follows by classical arguments while for $n$ this can be shown as in \cite[Lemma~2.1]{WinklerLargedataGlobalGeneralized2015}.

    \item
      Arguably, an important feature of \emph{classical solutions} of \eqref{prob:fluid_incl} is that the first solution component conserves mass. 
      While this is not guaranteed by the upper estimate in \eqref{eq:sol_concept:mass_ineq},
      in the two-dimensional setting this follows immediately from \eqref{eq:basic_apriori} and \eqref{eq:eps_sea_0:n_l1}.
      On the other hand, for $N \ge 3$ we cannot show that \eqref{eq:eps_sea_0:n_l1} holds;
      Lemma~\ref{lm:nc_uniform_int} ``only'' asserts uniform integrability of $\neps f(\ce)$ and not of $\neps$ (and we also lack a uniform lower bound for $f(\ce)$).
      Thus, obtaining mass conservation for generalized solutions also in the higher dimensional setting appears to be a difficult task, which we leave to further research.
  \end{enumerate}
\end{remark}

\subsection{Global classical solutions to regularized problems}\label{sec:ge_approx}
At the end of the present section, we construct global generalized solutions of \eqref{prob:fluid_incl} as the limit of global solutions to
\begin{align}\label{prob:approx}
  \begin{cases}
    \net + \ue \cdot \nabla \neps = \Delta \neps - \chi \nabla \cdot (\frac{\neps}{(1+\eps \neps) \ce} \nabla \ce) & \text{in $\Omega \times (0, \infty)$}, \\
    \cet + \ue \cdot \nabla \ce = \Delta \ce - \neps f(\ce)                                                        & \text{in $\Omega \times (0, \infty)$}, \\
    \uet + (\ue \cdot \nabla) \ue = \Delta \ue + \nabla P_\eps + \neps \nabla \phi, \quad \nabla \cdot \ue = 0     & \text{in $\Omega \times (0, \infty)$}, \\
    \partial_\nu \neps = \partial_\nu \ce = 0, \, \ue= 0                                                           & \text{on $\partial \Omega \times (0, \infty)$}, \\
    (\neps, \ce, \ue)(\cdot, 0) = (\nne, \cne, \une)                                                               & \text{in $\Omega$}.
  \end{cases}
\end{align}
for appropriate  $\nne$, $\cne$ and $\une$.
For sufficiently regular approximative initial data, the regularization term in the first equation guarantees that such global classical solutions do indeed exist.
\begin{lemma}\label{lm:global_ex_approx}
  For each $\eps \in (0, 1)$, there are $q > N$, $\beta \in (\frac12, 1)$,
  \begin{align}\label{eq:init:def_nc}
    \begin{cases}
      \nne \in \con0   \quad \text{with $\nne > 0$ in $\Ombar$ and $\intom \neps = \intom n_0$}, \\
      \cne \in \sob1q  \quad \text{with $\infc \le \cne \le \|c_0\|_{\leb\infty}$ in $\Ombar$}
    \end{cases}
  \end{align}
  and
  \begin{align}\label{eq:init:def_u}
    \une \begin{cases}
      \equiv 0 & \text{if $u_0 \equiv 0$}, \\
      \in \mc D(A^\beta) \quad \text{with $\intom |\une|^2 \le \intom |u_0|^2 +1$} & \text{if $u_0 \not\equiv 0$}
    \end{cases}
  \end{align}
  such that
  \begin{alignat}{2}
    (\nne, \ln \nne) &\to (n_0, \ln n_0) &&\qquad \text{in $(\leb1)^2$ and a.e.\ in $\Omega$}, \label{eq:init_approx:n}\\
    \cne &\to c_0 &&\qquad \text{in $\leb2$ and a.e.\ in $\Omega$}, \label{eq:init_approx:c}\\
    \une &\to u_0 &&\qquad \text{in $L^2(\Omega; \R^N)$ and a.e.\ in $\Omega$} \label{eq:init_approx:u}
  \end{alignat}
  as $\eps \sea 0$.
  Moreover, for each $\eps \in (0, 1)$, there exists a global classical solution
  \begin{align*}
    (\neps, \ce, \ue, P_\eps) \in \left( C^0(\Ombar \times [0, \infty) \cap C^{2, 1}(\Ombar \times (0, \infty) \right)^{1+1+N} \times C^{1, 0}(\Ombar \times (0, \infty))
  \end{align*}
  of \eqref{prob:approx} with $\ue \equiv 0$ and $P_\eps \equiv 0$ if $u_0 \equiv 0$ and $\phi \equiv 0$.
  This solution further satisfies $\neps, \ce > 0$ in $\Ombar \times (0, \infty)$,
  \begin{align}\label{eq:basic_apriori}
    \intom \neps = \intom n_0
    \quad \text{and} \quad
    \ce \le \|c_0\|_{\leb\infty}.
  \end{align}
\end{lemma}
\begin{proof}
  Since $(n_0, c_0, u_0)$ satisfy \eqref{eq:intro:cond_init},
  the existence of such an approximating family of initial data can be seen by typical convolution arguments for the first two components
  and by methods as in the proof of \cite[Theorem~III.4.1]{GaldiIntroductionMathematicalTheory2011} for the last one.
  
  For each $\eps > 0$,
  the existence of local maximal classical solutions to \eqref{prob:approx} can then be seen as in \cite[Lemma 2.1]{WinklerGlobalLargedataSolutions2012} and
  arguments similar to those in \cite[Lemma~2.2]{LiuLargetimeBehaviorTwodimensional2021} show that these solutions exist globally in time.
  (We also remark that \eqref{eq:intro:n=2_or_no_fluid} implies that we only need to obtain global classical solutions to the two-dimensional Navier--Stokes equation;
  for $N \ge 3$ we just set $\ue \equiv 0$ and $P_\eps \equiv 0$.)

  Finally, the strong maximum principle implies positivity of both $\neps$ and $\ce$ as well as the upper bound for $\ce$ in \eqref{eq:basic_apriori},
  while integrating the first equation in \eqref{prob:approx} and using \eqref{eq:init:def_nc} gives $\intom \neps(\cdot, t) = \intom \nne = \intom n_0$ and thus the first assertion in \eqref{eq:basic_apriori}.
\end{proof}

Similar as for instance in \cite[(2.15)]{WinklerTwodimensionalKellerSegel2016} or \cite[(2.21)]{LiuLargetimeBehaviorTwodimensional2021}, for $\eps \in (0, 1)$ we introduce the transformation 
\begin{align}\label{eq:def_w}
  \we \defs - \ln\frac{\ce}{\|c_0\|_{\leb\infty}},
\end{align}
which helps to streamline further arguments.
\begin{lemma}\label{lm:we}
  For $\eps \in (0, 1)$, the function given by \eqref{eq:def_w} fulfills
  \begin{align}\label{eq:w_eg}
    \we \ge 0, \quad
    \nabla \we = - \frac{\nabla \ce}{\ce}
    \quad \text{and} \quad
    \wet + \ue \cdot \nabla \we = \Delta \we - |\nabla \we|^2 + \neps \ce^{-1} f(\ce) 
  \end{align}
  in $\Omega \times (0, \infty)$
  and the first equation in \eqref{prob:approx} is equivalent to
  \begin{align}\label{eq:n_eq_w}
    \net + \ue \cdot \nabla \neps = \Delta \neps + \chi \nabla \cdot \left(\frac{\neps}{(1+\eps \neps)} \nabla \we\right) \qquad \text{in $\Omega \times (0, \infty)$}.
  \end{align}
\end{lemma}
\begin{proof}
  Nonnegativity of $\we$ follows from \eqref{eq:basic_apriori} and all remaining statements become evident after direct computations.
\end{proof}

For the remainder of this section,
we fix a family of approximate initial data $(\nne, \cne, \une)_{\eps \in (0, 1)}$ and the family of corresponding global classical solutions $(\neps, \ce, \ue, P_\eps)_{\eps \in (0, 1)}$ of \eqref{prob:approx}
as given by Lemma~\ref{lm:global_ex_approx}, define $\we$ as in \eqref{eq:def_w} and set $\wne \defs \wne(\cdot, 0) = - \ln\frac{\cne}{\|c_0\|_{\leb\infty}}$ for $\eps \in (0, 1)$.
Moreover, we will make use of both statements in \eqref{eq:basic_apriori} quite often without always explicitly referring to \eqref{eq:basic_apriori}.

\subsection{A~priori estimates and convergence to solution candidates}\label{sec:apriori}
A well-known quasi-energy functional for chemotaxis systems with consumption and logarithmic sensitivity (cf.\ for instance \cite[Lemma~2.2 and Lemma~2.3]{WinklerTwodimensionalKellerSegel2016})
allows us to gain $\eps$-independent a~priori estimates going significantly beyond \eqref{eq:basic_apriori}.
\begin{lemma}\label{lm:quasi_energy}
  There exists $\newgc{quasi_energy} > 0$ such that
  \begin{align}\label{eq:quasi_energy:ddt}
        \ddt \left( - \intom \ln \neps + \chi^2 \intom \we \right)
        + \frac12 \intom \frac{|\nabla \neps|^2}{\neps^2} + \frac{\chi^2}{2} \intom |\nabla \we|^2
    \le \chi^2 \intom \neps \ce^{-1} f(\ce) 
  \end{align}
  in $(0, \infty)$ for all $\eps \in (0, 1)$ and
  \begin{align}\label{eq:quasi_energy:integrated}
        \sup_{t \in (0, T)} \intom \we(\cdot, t)
        + \intntom \frac{|\nabla \neps|^2}{\neps^2} + \intntom |\nabla \we|^2
    \le \gc{quasi_energy} (T+1)
  \end{align}
  for all $T > 0$ and all $\eps \in (0, 1)$.
\end{lemma}
\begin{proof}
  Testing \eqref{eq:n_eq_w} with $-\frac1{\neps}$ and the last equation in \eqref{eq:w_eg} with $\we$ yields
  \begin{align*}
          - \ddt \intom \ln \neps
    &=    - \intom \frac{|\nabla \neps|^2}{\neps^2} - \chi \intom \frac{\nabla \neps \cdot \nabla \we}{\neps (1+\eps \neps)}
     \le  - \frac12 \intom \frac{|\nabla \neps|^2}{\neps^2} + \frac{\chi^2}{2} \intom |\nabla \we|^2
  \end{align*}
  and
  \begin{align*}
          \ddt \intom \we
    &=    - \intom |\nabla \we|^2 + \intom \neps \ce^{-1} f(\ce)
  \end{align*}
  in $\Omega \times (0, \infty)$ for all $\eps \in (0, 1)$, where the contribution from the fluid vanishes due to $\nabla \cdot \ue = 0$ in $\Omega \times (0, \infty)$.
  Adding suitable multiples of these estimates already gives \eqref{eq:quasi_energy:ddt}
  and since $f(s) \le s \|f'\|_{L^\infty((0, \|c_0\|_{\leb\infty}))}$ for $s \in [0, \|c_0\|_{\leb\infty}]$ by \eqref{eq:intro:cond_f},
  \eqref{eq:basic_apriori} implies that the right-hand side in \eqref{eq:quasi_energy:ddt} is bounded in $(0, \infty)$ independently of $\eps$.
  As moreover $\intom \ln \neps \le \intom \neps = \intom n_0$ in $(0, \infty)$ by \eqref{eq:basic_apriori} and $(\intom \wne)_{\eps \in (0, 1)}$ is bounded by \eqref{eq:init:def_nc},
  integrating \eqref{eq:quasi_energy:ddt} in time shows that \eqref{eq:quasi_energy:integrated} holds for some $\gc{quasi_energy} > 0$.
\end{proof}

The usual Navier--Stokes energy functional, Lemma~\ref{lm:heihoff_ineq} and Lemma~\ref{lm:quasi_energy} imply a~priori estimates also for the fluid.
\begin{lemma}\label{lm:fluid_energy}
  There exists $\newgc{fluid_n_ln_n} > 0$ such that
  \begin{align}\label{eq:fluid_energy:ddt}
    \ddt \intom |\ue(\cdot, t)|^2 + \intom |\nabla \ue|^2 \le \gc{fluid_n_ln_n} \intom \frac{|\nabla \neps|^2}{\neps^2}
    \qquad \text{for all $\eps \in (0, 1)$}.
  \end{align}
  In particular,
  \begin{align}\label{eq:fluid_energy:nabla_u_l2_bdd}
    \intntom |\nabla \ue|^2 \le \intom |u_0|^2 + 1 + \gc{quasi_energy} \gc{fluid_n_ln_n} (T+1)
    \qquad \text{for all $T > 0$ and all $\eps \in (0, 1)$},
  \end{align}
  where $\gc{quasi_energy}$ is as in Lemma~\ref{lm:quasi_energy}.
\end{lemma}
\begin{proof}
  If $\ue \equiv 0$, these estimates are trivial.
  Else, \eqref{eq:intro:n=2_or_no_fluid} implies $N=2$ so that we can apply Heihoff's inequalities \eqref{eq:heihoff_ineq:1} and \eqref{eq:heihoff_ineq:2},
  which allow us to argue as in \cite[Lemma~3.4]{HeihoffTwoNewFunctional2022} to obtain \eqref{eq:fluid_energy:ddt}.
  Thereupon, \eqref{eq:fluid_energy:nabla_u_l2_bdd} follows from an integration in time, \eqref{eq:init:def_u} and \eqref{eq:quasi_energy:integrated}.
\end{proof}

Finally, we also obtain bounds for the time derivatives.
\begin{lemma}\label{lm:time_est}
  Let $T > 0$. There is $\newgc{time_est} > 0$ such that
  \begin{align*}
      \|(\ln \neps+1)_t\|_{L^1((0, T); \dual{\sob N2})} 
    + \|(\ln \ce)_t\|_{L^1((0, T); \dual{\sob N2})} 
    + \|\uet\|_{L^1((0, T); \dual{\sob N2 \cap W_{0, \sigma}^{1,2}(\Omega)})}
    &\le \gc{time_est}.
  \end{align*}  
\end{lemma}
\begin{proof}
  Straightforward testing procedures yield $\newlc{factor} > 0$ such that
  \begin{align*}
          \left| \intom (\ln(\neps+1))_t \varphi \right|
    &\le  \lc{factor} \left(1 + \intom \frac{|\nabla \neps|^2}{(\neps+1)^2} + \intom |\nabla \we|^2 + \intom |\ue|^2 \right)
          \left(\|\varphi\|_{\leb\infty} + \|\nabla \varphi\|_{\leb2} \right), \\
          \left| \intom (\ln \ce)_t \varphi \right|
    &\le  \lc{factor} \left(1 + \intom |\nabla \we|^2 + \intom |\ue|^2 + \intom \neps \right)
          \left(\|\varphi\|_{\leb\infty} + \|\nabla \varphi\|_{\leb2} \right) \quad \text{and} \\
          \left| \intom \uet \psi \right|
    &\le  \lc{factor} \left(1 + \intom |\ue|^2 + \intom |\nabla \ue|^2 + \intom \neps \right)
          \left(\|\psi\|_{\leb\infty} + \|\nabla \psi\|_{\leb2} \right)
  \end{align*}
  in $(0, \infty)$ 
  for all $\varphi \in C^\infty(\Ombar)$, $\psi \in C_c^\infty(\Omega; \R^N)$ with $\nabla \cdot \psi = 0$ and all $\eps \in (0, 1)$;
  for details we refer to \cite[Lemma~2.4]{WangGlobalLargedataGeneralized2016} and \cite[Lemma~2.12]{WinklerSmallmassSolutionsTwodimensional2020}.
  Since $\sob N2 \embed \leb\infty$, we thus obtain the desired estimates
  by integrating in time and applying Lemma~\ref{lm:quasi_energy}, Lemma~\ref{lm:fluid_energy} and the Poincar\'e inequality.
\end{proof}

The a priori estimates obtained in the lemmata above make several compactness results applicable, allowing us to construct solution candidates as limits in appropriate spaces of the solutions to \eqref{prob:approx}.
\begin{lemma}\label{lm:eps_sea_0}
  There exist a triple of functions
  \begin{align*}
    (n, c, u) \in L^\infty((0, \infty); \leb1) \times L^2_{\loc}([0, \infty); \sob12) \times L^2_{\loc}([0, \infty); W_{0, \sigma}^{1, 2}(\Omega))
  \end{align*}
  with $n \ge 0$ and $0 \le c \le \|c_0\|_{\leb\infty}$ a.e.\ in $\Omega \times (0, \infty)$
  and a null sequence $(\eps_j)_{j \in \N} \subset(0, 1)$ such that
  \begin{alignat}{2}
    \neps                              & \ra n               && \qquad \text{a.e.\ in $\Omega \times (0, \infty)$}, \label{eq:eps_sea_0:n_pw} \\
    \neps(\cdot, t)                    & \ra n(\cdot, t)     && \qquad \text{a.e.\ in $\Omega$ for a.e.\ $t \in (0, \infty)$}, \label{eq:eps_sea_0:n_t_pw} \\
    \ln(\neps+1)                       & \ra \ln(n+1)        && \qquad \text{in $L_{\loc}^2(\Ombarinf)$ and a.e.\ in $\Omega \times (0, \infty)$}, \label{eq:eps_sea_0:ln_n_l2} \\
    \nabla \ln(\neps+1)                & \rh \nabla \ln(n+1) && \qquad \text{in $L_{\loc}^2(\Ombarinf)$}, \label{eq:eps_sea_0:nabla_ln_n_l2} \\
    \ce                                & \ra c               && \qquad \text{in $L_{\loc}^2(\Ombarinf)$ and a.e.\ in $\Omega \times (0, \infty)$}, \label{eq:eps_sea_0:c_l2} \\
    \nabla \ce                         & \rh \nabla c        && \qquad \text{in $L_{\loc}^2(\Ombarinf)$}, \label{eq:eps_sea_0:nabla_c_l2} \\
    \ln \ce                            & \ra \ln c           && \qquad \text{in $L_{\loc}^2(\Ombarinf)$ and a.e.\ in $\Omega \times (0, \infty)$}, \label{eq:eps_sea_0:ln_c_l2} \\
    \nabla \ln \ce                     & \rh \nabla \ln c    && \qquad \text{in $L_{\loc}^2(\Ombarinf)$}, \label{eq:eps_sea_0:nabla_ln_c_l2} \\
    \ue                                & \ra u               && \qquad \text{in $L_{\loc}^2(\Ombarinf)$ and a.e.\ in $\Omega \times (0, \infty)$}, \label{eq:eps_sea_0:u_l2} \\
    \nabla \ue                         & \rh \nabla u        && \qquad \text{in $L_{\loc}^2(\Ombarinf)$} \label{eq:eps_sea_0:nabla_u_l2}
  \intertext{and, if $N = 2$ also}
    \neps                              & \ra n               && \qquad \text{in $L_{\loc}^1(\Ombarinf)$} \label{eq:eps_sea_0:n_l1}
  \end{alignat}
  as $\eps = \eps_j \sea 0$.
  Furthermore, \eqref{eq:sol_concept:mass_ineq} and \eqref{eq:sol_concept:u_weak_sol} hold.
\end{lemma}
\begin{proof}
  For each $T > 0$, \eqref{eq:quasi_energy:integrated} and \eqref{eq:fluid_energy:nabla_u_l2_bdd} entail that the families $(\ln (\neps+1))_{\eps \in (0, 1)}, (\ln \ce)_{\eps \in (0, 1)}$ and $(\ue)_{\eps \in (0, 1)}$
  are bounded in $L^2((0, T); \sob12)$, $L^2((0, T); \sob12)$ and $L^2((0, T); W_{0, \sigma}^{1,2}(\Omega))$, respectively.
  As also their time derivatives are bounded in appropriate spaces by Lemma~\ref{lm:time_est},
  three applications of the Aubin--Lions lemma, multiple diagonalization arguments and the fact that convergence in $L^2$ implies a.e.\ convergence along a subsequence
  yield a null sequence $(\eps_j)_{j \in \N}$ and $(z_1, z_2, u) \in (L^2_{\loc}([0, \infty); \sob12))^2 \times L^2_{\loc}([0, \infty); W_{0, \sigma}^{1, 2}(\Omega))$
  such that \eqref{eq:eps_sea_0:ln_n_l2}, \eqref{eq:eps_sea_0:nabla_ln_n_l2} and \eqref{eq:eps_sea_0:ln_c_l2}--\eqref{eq:eps_sea_0:nabla_u_l2} hold with $(\ln(n+1), \ln c)$ replaced by $(z_1, z_2)$.
  Setting $n \defs \ure^{z_1}-1$ and $c \defs \ure^{z_2}$,
  we further obtain \eqref{eq:eps_sea_0:n_pw}, \eqref{eq:eps_sea_0:ln_n_l2}, \eqref{eq:eps_sea_0:nabla_ln_n_l2}, \eqref{eq:eps_sea_0:ln_c_l2} as well as \eqref{eq:eps_sea_0:nabla_ln_c_l2}
  and due to \eqref{eq:basic_apriori} and \eqref{eq:quasi_energy:integrated} also \eqref{eq:eps_sea_0:c_l2} and \eqref{eq:eps_sea_0:nabla_c_l2},
  while \eqref{eq:eps_sea_0:n_t_pw} follows from \eqref{eq:eps_sea_0:n_pw}.
  The upper and lower bound for $n$ and $c$ are then direct consequences of \eqref{eq:eps_sea_0:n_pw} and \eqref{eq:eps_sea_0:c_l2}
  and the corresponding bounds for the approximate solutions provided by Lemma~\ref{lm:global_ex_approx}.

  If $N=2$, the estimates \eqref{eq:quasi_energy:integrated} and \eqref{eq:heihoff_ineq:2} together with the de la Vall\'ee Poussin theorem
  assert that $(\neps)_{\eps \in (0, 1)}$ is uniformly integrable in $\Omega \times (0, T)$ for all $T > 0$,
  so that after switching to another subsequence Vitali's theorem gives \eqref{eq:eps_sea_0:n_l1}.

  Moreover, by Fatou's lemma, \eqref{eq:eps_sea_0:n_t_pw} and the first statement in \eqref{eq:basic_apriori} we have
  \begin{align*}
    \intom n(\cdot, t) \le \liminf_{j \to \infty} \intom n_{\eps_j}(\cdot, t) = \intom n_0
    \qquad \text{for a.e.\ $t \in (0, \infty)$},
  \end{align*}
  which implies \eqref{eq:sol_concept:mass_ineq} and that $n$ belongs to $L^\infty((0, \infty); \leb1)$.
  Finally, the weak formulation \eqref{eq:sol_concept:u_weak_sol} for the fluid equation follows from taking the limit in each term of the corresponding equations for the approximate problems,
  which is possible due to \eqref{eq:init_approx:u}, \eqref{eq:eps_sea_0:u_l2}, \eqref{eq:eps_sea_0:nabla_u_l2}, \eqref{eq:eps_sea_0:n_l1} and \eqref{eq:intro:n=2_or_no_fluid}.
\end{proof}

\subsection{Strong convergence of \tops{$\neps f(\ce)$}{n\_eps f(c\_eps)} and \tops{$\nabla \we$}{grad w\_eps}}\label{sec:strong_conv}
The convergence properties asserted by Lemma~\ref{lm:eps_sea_0} are yet insufficient
to take the limit in the corresponding versions of both \eqref{eq:sol_concept:n_ln_supersol} and \eqref{eq:sol_concept:v_weak_sol} for the approximate solutions.
Indeed, the latter one needs weak convergence of $\neps f(\ce)$ in $L_{\loc}^1(\Ombarinf)$ 
while for the former the critical term is $-\chi \intninfom \frac{n \nabla n \cdot \nabla c}{(n+1)^2 c} \varphi$,
which requires convergence of the product of two gradient terms and thus (for instance) \emph{strong} convergence of $\nabla \we$.

As discussed in \cite{FuestStrongConvergenceWeighted2022} (and already mentioned in the introduction), these issues are related:
At least in the fluid-free case, weak convergence of the source term in heat equations implies strong convergence of certain weighted gradients of the corresponding solution.

This leaves the question how to obtain sufficiently strong a~priori estimates allowing us to conclude convergence of $\neps f(\ce)$ in $L_{\loc}^1(\Ombarinf)$.
Here we follow an idea recently developed in \cite[Proposition~4.5]{HombergEtAlExistenceGeneralizedSolutions2022}: to consider the functional $\intom (\ln \neps) \ce$.
When calculating its time derivative, only a single term with favourable sign appears (stemming from the consumption term in the second equation in \eqref{prob:approx})
but all remaining terms (in particular those involving gradients) fortunately are already known to be bounded thanks to \eqref{eq:quasi_energy:integrated}.
\begin{lemma}\label{lm:nc_uniform_int}
  Let $T > 0$. Then there is $\newgc{nc_uniform_int} > 0$ such that
  \begin{align}\label{eq:nc_uniform_int:statement}
    \intntom \neps f(\ce) |\ln (\neps f(\ce))| \le \gc{nc_uniform_int}
    \qquad \text{for all $\eps \in (0, 1)$}.
  \end{align}
\end{lemma}
\begin{proof}
  We set $\newlc{c_infty} \defs \|c_0\|_{\leb\infty}$ and calculate
  \begin{align*}
          \ddt \intom (\ln \neps) \ce
    &=    \intom \net \frac{\ce}{\neps} 
          + \intom \cet \ln \neps \\
    &=    - \intom \left(\nabla \neps - \frac{\chi \neps \nabla \ce}{(1+\eps\neps)\ce} \right)
            \cdot \left(-\frac{\ce \nabla \neps}{\neps^2} + \frac{\nabla \ce}{\neps} \right)
          - \intom \frac{\nabla \neps \cdot \nabla \ce}{\neps}
          - \intom \neps f(\ce) \ln \neps \\
    &=    \intom \frac{\ce |\nabla \neps|^2}{\neps^2}
          - 2 \intom \frac{\nabla \neps \cdot \nabla \ce}{\neps}
          - \chi \intom \frac{\nabla \neps \cdot \nabla \ce}{\neps(1+\eps\neps)}
          + \chi \intom \frac{|\nabla \ce|^2}{(1+\eps\neps)\ce}
          - \intom \neps f(\ce) \ln \neps \\
    &\le  \left(\lc{c_infty} + 1 + \frac{\chi}{2} \right) \intom \frac{|\nabla \neps|^2}{\neps^2}
          + \left( \lc{c_infty}^2 + \frac{\chi \lc{c_infty}^2}{2} + \chi \lc{c_infty} \right) \intom \frac{|\nabla \ce|^2}{\ce^2}
          - \intom \neps f(\ce) \ln \neps
  \end{align*}
  in $(0, \infty)$ for all $\eps \in (0, 1)$,
  so that drawing on \eqref{eq:quasi_energy:integrated}, \eqref{eq:init:def_nc} and \eqref{eq:init_approx:n}, we can find $\newlc{nc_ln_n_bdd} > 0$ such that
  \begin{align*}
    \intntom \neps f(\ce) \ln \neps \le \lc{nc_ln_n_bdd}
    \qquad \text{for all $\eps \in (0, 1)$}.
  \end{align*}
  As moreover
  \begin{align*}
          \intntom \neps f(\ce) \ln f(\ce)
    &\le  \intntom \neps f(\ce)^2
     \le  \|f\|_{L^\infty((0, \lc{c_infty}))}^2 T \intom n_0
    \sfed \newlc{nc_ln_c_bdd}
    \qquad \text{for all $\eps \in (0, 1)$}
  \end{align*}
  by \eqref{eq:basic_apriori}, we can conclude that
  \begin{align*}
        \intntom \neps f(\ce) \ln(\neps f(\ce))
    =   \intntom \neps f(\ce) (\ln \neps  + \ln f(\ce))
    \le \lc{nc_ln_n_bdd} + \lc{nc_ln_c_bdd}
    \qquad \text{for all $\eps \in (0, 1)$}.
  \end{align*}
  Since $s |\ln s| \le \frac1\ure$ for $s \in (0, 1)$, this implies \eqref{eq:nc_uniform_int:statement}
  for $\gc{nc_uniform_int} \defs \frac{2|\Omega| T}{\ure} + \lc{nc_ln_n_bdd} + \lc{nc_ln_c_bdd}$.
\end{proof}

The a~priori estimate \eqref{eq:nc_uniform_int:statement} now allows us to indeed obtain convergence properties
guaranteeing that also \eqref{eq:sol_concept:v_weak_sol} is fulfilled by the functions constructed in Lemma~\ref{lm:eps_sea_0}.
\begin{lemma}\label{lm:strong_conv}
  Let $n$, $c$ and $(\eps_j)_{j \in \N}$ be as given by Lemma~\ref{lm:eps_sea_0}.
  Then there exists a subsequence of $(\eps_j)_{j \in \N}$, which we do not relabel, such that
  \begin{alignat}{2}
    \neps f(\ce) & \ra n f(c) && \qquad \text{in $L_{\loc}^1(\Ombarinf)$}, \label{eq:strong_conv:nc_l1} \\
    \nabla \ce & \ra \nabla c && \qquad \text{in $L_{\loc}^2(\Ombarinf)$}, \label{eq:strong_conv:nabla_c_l2}
  \end{alignat}
  as $\eps = \eps_j \sea 0$.
  Moreover, \eqref{eq:sol_concept:v_weak_sol} holds.
\end{lemma}
\begin{proof}
  According to Lemma~\ref{lm:nc_uniform_int} and the de la Vall\'ee Poussin theorem,
  $(\neps f(\ce))_{\eps \in (0, 1)}$ is uniformly integrable in $\Omega \times (0, T)$ for all $T \in (0, \infty)$,
  so that Vitali's theorem, \eqref{eq:eps_sea_0:n_pw}, \eqref{eq:eps_sea_0:c_l2} and a diagonalization argument
  imply that there exists a subsequence of $(\eps_j)_{j \in \N}$ such that \eqref{eq:strong_conv:nc_l1} holds.
  Combined with \eqref{eq:init_approx:c} and Lemma~\ref{lm:eps_sea_0},
  this shows that we can pass to the limit in the weak formulation of the second equation in \eqref{prob:approx}; that is, \eqref{eq:sol_concept:v_weak_sol} holds.

  By testing \eqref{eq:sol_concept:v_weak_sol} with regular functions approximating $c$,
  one can obtain $\intntom |\nabla c|^2 \ge \limsup_{j \to \infty} \intntom |\nabla c_{\eps_j}|^2$ for a.e.\ $T > 0$,
  see for instance \cite[Subsection~4.2]{WinklerGlobalMasspreservingSolutions2018} for details.
  In conjunction with \eqref{eq:eps_sea_0:nabla_c_l2}, this then gives \eqref{eq:strong_conv:nabla_c_l2}.
  (Let us remark that in the fluid-free case, \eqref{eq:strong_conv:nabla_c_l2} alternatively also directly follows from \cite[Theorem~1.1]{FuestStrongConvergenceWeighted2022}.)
\end{proof}

Although \eqref{eq:strong_conv:nabla_c_l2} is an important step towards showing that also \eqref{eq:sol_concept:n_ln_supersol} holds,
appropriate convergence of the critical term $-\chi \frac{\neps \nabla \neps \cdot \nabla \ce}{(\neps+1)^2(1+\eps \neps) \ce}$ seems to require strong convergence of $\nabla \ln \ce$, not only of $\nabla \ce$.
Unfortunately, combining the estimate for $\nabla \ln \ce$ contained in \eqref{eq:quasi_energy:integrated} and the convergence asserted in \eqref{eq:strong_conv:nabla_c_l2}
is just about insufficient to prove the desired strong $L^2$ convergence of $\nabla \ln \ce$.

In the two-dimensional setting, one can rely on \eqref{eq:eps_sea_0:n_l1}
to show that $\we$ converges to a function solving the formal limit of the partial differential equation appearing in \eqref{eq:w_eg} in a weak sense
and then test this limit equation with suitably chosen Steklov averages in order to eventually obtain strong convergence of $\nabla \we$;
we refer to \cite[Lemma~2.9 and Lemma~2.10]{WangGlobalLargedataGeneralized2016} for details in a closely related setting.

In the higher dimensional setting however, we do not have \eqref{eq:eps_sea_0:n_l1} at our disposal and hence need to argue differently.
The general strategy consists of deriving and making use of stronger bounds for $\nabla \ln \ce$ than those provided by \eqref{eq:quasi_energy:integrated},
a concept discussed in detail in \cite[Section~4]{FuestStrongConvergenceWeighted2022}.
Crucially making use of \eqref{eq:ev_smooth:cond_f}, we can indeed obtain such estimates.
\begin{lemma}\label{lm:nabla_ln_c_eta}
  Let $T > 0$ and $\eta \in (0, 1)$.
  Then there exists $A > 0$ such that
  \begin{align}\label{eq:nabla_ln_c_eta:statement}
    \int_0^T \intom \mathds 1_{\{\ce \le A\}} \frac{|\nabla \ce|^2}{\ce^2} \le \eta
    \qquad \text{for all $\eps \in (0, 1)$}.
  \end{align}
\end{lemma}
\begin{proof}
  We recall that $\cne \ge \delta$ by \eqref{eq:init:def_nc} for all $\eps \in (0, 1)$, abbreviate
  \begin{align*}
    \newlc{c0_linfty} \defs \|c_0\|_{\leb\infty}
    \quad \text{and} \quad
    \newlc{bdd} \defs \left\|\frac{ \lc{c0_linfty}|\Omega|}f\right\|_{L^\infty((\infc, \lc{c0_linfty}))} + T \intom n_0,
  \end{align*}
  and choose $s_0 > 0$ so small that $f' \ge - \frac{\eta}{4 \lc{bdd}}$ in $[0, s_0]$.
  This allows us to fix a function $g \in C^1([0, \infty))$ with $g \ge f$ and $g' > 0$ in $(s_0, \infty)$ as well as
  \begin{align*}
    g(s) = f(s) + \tfrac{\eta s}{2 \lc{bdd}} \qquad \text{for $s \in [0, s_0]$}.
  \end{align*}
  Then $g \ge f$ and $g' > 0$ in $(0, \infty)$, and $G(s) \defs - \int_{\lc{c0_linfty}}^s \frac{\dsigma}{g(\sigma)}$ is nonnegative for all $s \in (0, \lc{c0_linfty}]$.
  By testing the second equation in \eqref{prob:approx} with $-\frac1g$, we obtain
  \begin{align}\label{eq:nabla_ln_c_eta:ddt}
        \ddt \intom G(\ce)
    =   \intom \nabla \ce \cdot \nabla \frac1{g(\ce)} + \intom \frac{f(\ce) \neps}{g(\ce)}
    \le - \intom \frac{g'(\ce)}{g^2(\ce)} |\nabla \ce|^2 + \intom n_0
  \end{align}
  in $(0, \infty)$ for all $\eps \in (0, 1)$.
  As moreover
  \begin{align*}
      \intom G(\cne)
    = - \intom \int_{\lc{c0_linfty}}^{\cne(x)} \frac{1}{g(s)} \ds \dx
    \le |\Omega| (\lc{c0_linfty}-\infc) \sup_{s \in [\infc, \lc{c0_linfty}]} \frac{1}{g(s)}
    \le \lc{c0_linfty} |\Omega| \sup_{s \in [\infc, \lc{c0_linfty}]} \frac{1}{f(s)}
  \end{align*}
  by \eqref{eq:init:def_nc},
  integrating \eqref{eq:nabla_ln_c_eta:ddt} over $(0, T)$ yields
  \begin{align*}
        \intom G(\ce(\cdot, T)) + \intntom  \frac{g'(\ce)}{g^2(\ce)} |\nabla \ce|^2
    \le \intom G(\cne) + T \intom n_0
    \le \lc{bdd}.
  \end{align*}
  Since
  \begin{align*}
      \lim_{s \sea 0} \frac{s^2 g'(s)}{g^2(s)}
    = \lim_{s \sea 0} \left( \frac{s-0}{g(s)-g(0)} \right)^2 g'(s)
    = \frac{g'(0)}{(g'(0))^2} = \frac{2\lc{bdd}}{\eta}
  \end{align*}
  by \eqref{eq:ev_smooth:cond_f}, there is $A > 0$ such that $\frac{g'(s)}{g^2(s)} \ge \frac{\lc{bdd}}{\eta s^2}$ for all $s \in (0, A)$.
  Thus, we conclude that
  \begin{align*}
        \frac{\lc{bdd}}{\eta} \intntom \mathds 1_{\{\ce \le A\}} \frac{|\nabla \ce|^2}{\ce^2}
    \le \intntom \mathds 1_{\{\ce \le A\}} \frac{g'(\ce)}{g^2(\ce)} |\nabla \ce|^2
    \le \intntom \frac{g'(\ce)}{g^2(\ce)} |\nabla \ce|^2
    \le \lc{bdd}
  \end{align*}
  for all $\eps \in (0, 1)$,
  which entails \eqref{eq:nabla_ln_c_eta:statement}.
\end{proof}

Similar to \cite[Lemma~3.10]{FuestStrongConvergenceWeighted2022}, \eqref{eq:strong_conv:nabla_c_l2} and \eqref{eq:nabla_ln_c_eta:statement} now imply strong convergence of $\nabla \ln \ce$.
\begin{lemma}\label{lm:strong_conv2}
  Let $n$, $c$ and $(\eps_j)_{j \in \N}$ be as given by Lemma~\ref{lm:strong_conv}.
  Then (without switching to a subsequence)
  \begin{align}
    \nabla \ln \ce & \ra \nabla \ln c \qquad \text{in $L_{\loc}^2(\Ombarinf)$ as $\eps = \eps_j \sea 0$}. \label{eq:strong_conv2:nabla_w_l2}
  \end{align}
\end{lemma}
\begin{proof}
  We fix $T > 0$.
  For arbitrary $\eta > 0$, we let $A > 0$ be as given by Lemma~\ref{lm:nabla_ln_c_eta}.
  Moreover, we fix $\xi_1, \xi_2 \in C^\infty([0, \infty); [0, 1])$ with $\supp \xi_1 \subset [0, A]$, $\supp \xi_2 \subset [\frac A2, \infty)$ and $\xi_1 + \xi_2 = 1$ on $[0, \infty)$.
  As $(\xi_2(\cej) \cej^{-2})_{j \in \N}$ is bounded and converging pointwise a.e.\ to $\xi_2(c)c^{-2}$ due to continuity of $\xi_2$ and \eqref{eq:eps_sea_0:c_l2},
  Lebesgue's theorem and \eqref{eq:strong_conv:nabla_c_l2} show that $\sqrt{\xi_2(\cej)} \cej^{-1} \nabla \cej \to \sqrt{\xi_2(c)} c^{-1} \nabla c$ in $L^2(\Omega \times (0, T))$ as $j \to \infty$.
  In combination with \eqref{eq:nabla_ln_c_eta:statement}, this asserts
  \begin{align*}
          \limsup_{j \to \infty} \intntom \frac{|\nabla \cej|^2}{\cej^2}
    &\le  \limsup_{j \to \infty} \intntom \xi_1(\cej) \frac{|\nabla \cej|^2}{\cej^2} + \limsup_{j \to \infty} \intntom \xi_2(\cej) \frac{|\nabla \cej|^2}{\cej^2} \\
    &\le  \limsup_{j \to \infty} \intntom \mathds 1_{\{\cej \le A\}} \frac{|\nabla \cej|^2}{\cej^2} + \intntom \xi_2(c) \frac{|\nabla c|^2}{c^2} \\
    &\le  \eta + \intntom \frac{|\nabla c|^2}{c^2}.
  \end{align*}
  As $\eta$ was chosen arbitrary, we conclude
  \begin{align*}
          \limsup_{j \to \infty} \intntom \frac{|\nabla \cej|^2}{\cej^2}
    &\le  \intntom \frac{|\nabla c|^2}{c^2},
  \end{align*}
  which together with \eqref{eq:eps_sea_0:nabla_ln_c_l2} implies \eqref{eq:strong_conv2:nabla_w_l2}.
\end{proof}

Crucially relying on \eqref{eq:strong_conv2:nabla_w_l2}, we can finally show that the remaining bullet point in Definition~\ref{def:sol_concept} is also fulfilled;
that is, we verify that the function $n$ given by Lemma~\ref{lm:eps_sea_0} is a weak $\ln$-supersolution of the corresponding equation.
\begin{lemma}\label{lm:u_ln_supersol}
  Let $(n, c, u)$ and $(\eps_j)_{j \in \N}$ be as given by Lemma~\ref{lm:eps_sea_0} and Lemma~\ref{lm:strong_conv}, respectively.
  Then \eqref{eq:sol_concept:n_ln_supersol} holds.
\end{lemma}
\begin{proof}
  Testing the first equation in \eqref{prob:approx} with $\frac{\varphi}{\neps+1}$ for henceforth fixed $0 \le \varphi \in C_c^\infty(\Ombarinf)$ gives 
  \begin{align*}
    &\pe  - \intninfom \ln(\neps+1) \varphi_t
          - \intom \ln(\nne+1) \varphi(\cdot, 0) \\
    &=    \intninfom \frac{|\nabla \neps|^2}{(\neps+1)^2} \varphi
          - \intninfom \frac{\nabla \neps \cdot \nabla \varphi}{\neps+1} \notag \\
    &\pe  - \chi \intninfom \frac{\neps \nabla \neps \cdot \nabla \ce}{(\neps+1)^2(1+\eps\neps)\ce} \varphi
          + \chi \intninfom \frac{\neps \nabla \ce \cdot \nabla \varphi}{(\neps+1)\ce} \notag \\
    &\pe  + \intninfom \ln(\neps+1) (\ue \cdot \nabla \varphi)
    \qquad \text{for all $\eps \in (0, 1)$}
  \end{align*}
  and the convergence properties asserted by \eqref{eq:init_approx:n}, Lemma~\ref{lm:eps_sea_0} and Lemma~\ref{lm:strong_conv} allow us to take the limit in each term herein (at the cost of changing ``$=$'' to ``$\ge$'').
  For instance, \eqref{eq:eps_sea_0:nabla_ln_n_l2}, nonnegativity of $\varphi$ and the weak lower semicontinuity of the norm imply
  \begin{align*}
    \liminf_{j \to \infty} \intninfom \frac{|\nabla n_{\eps_j}|^2}{(n_{\eps_j}+1)^2} \varphi
    \ge \intninfom \frac{|\nabla n|^2}{(n+1)^2} \varphi,
  \end{align*}
  and 
  \begin{align*}
      \lim_{j \to \infty} \intninfom \frac{n_{\eps_j} \nabla n_{\eps_j} \cdot \nabla \cej}{(n_{\eps_j}+1)^2(1+\eps_jn_{\eps_j}) \cej} \varphi
    = \intninfom \frac{n \nabla n \cdot \nabla c}{(n+1)^2 c} \varphi
  \end{align*}
  follows from \eqref{eq:eps_sea_0:n_pw}, \eqref{eq:eps_sea_0:nabla_ln_n_l2} and \eqref{eq:strong_conv2:nabla_w_l2}.
\end{proof}

\subsection{Proof of Theorem~\ref{th:global_ex}}\label{sec:pf_th1}
All that remains to do regarding Theorem~\ref{th:global_ex} is to collect the statements proven in the lemmata above.
\begin{proof}
  That the triple $(n, c, u)$ constructed in Lemma~\ref{lm:eps_sea_0} is a global generalized solution in the sense of Definition~\ref{def:sol_concept} follows directly
  from Lemma~\ref{lm:eps_sea_0}, Lemma~\ref{lm:strong_conv} and Lemma~\ref{lm:u_ln_supersol}.
\end{proof}

\section{Eventual smoothness and stabilization in the two-dimensional setting}\label{sec:ev_smooth}
Apart from the assumptions made in the beginning of Section~\ref{sec:global_ex},
throughout this section we also assume $N=2$, i.e.\ that $\Omega$ is a planar domain.
Moreover, we again fix initial data $(\nne, \cne, \une)_{\eps \in (0, 1)}$ and the corresponding solutions $(\neps, \ce, \ue, P_\eps)$, $\eps \in (0, 1)$, of \eqref{prob:approx} given by Lemma~\ref{lm:global_ex_approx}
let $\we$ be as in \eqref{eq:def_w} and abbreviate $\wne \defs \we(\cdot, 0)$ for $\eps \in (0, 1)$.

Our goal is to prove Theorem~\ref{th:ev_smooth}, that is, that the global generalized solution $(n, c, u)$ given by Theorem~\ref{th:global_ex} eventually becomes smooth and stabilizes in the large-time limit.
To that end, we first construct an energy functional in Subsection~\ref{sec:cond_energy}
which is conditional in the sense that it requires the smallness of certain quantities at some time $t_0$ and then provides estimates in $\Omega \times (t_0, \infty)$.
Next, in Subsection~\ref{sec:ev_small} we show that such a $t_0$ indeed exists,
before making use of the bounds provided by the conditional energy functional in Subsection~\ref{sec:ev_smooth_large_time}
and proving Theorem~\ref{th:ev_smooth} and Theorem~\ref{eq:ev_smooth:cond_f} in Subsection~\ref{sec:proof_ev_smooth}.

As already mentioned in the introduction, the main difference compared to \cite{LiuLargetimeBehaviorTwodimensional2021}
is that we do not need to require smallness of $\intom n_0$.
Instead, we make use of the assumption \eqref{eq:ev_smooth:cond_f} both in Lemma~\ref{lm:cond_func} and Lemma~\ref{lm:nabla_ln_n_nabla_w_uniform_spacetime}
and of Heihoff's inequalities \eqref{eq:heihoff_ineq:1} and \eqref{eq:heihoff_ineq:2} in Lemma~\ref{lm:nabla_ln_n_nabla_w_uniform_spacetime} (and indirectly in Lemma~\ref{lm:cond_func_result} which refers to Lemma~\ref{lm:fluid_energy}).

\subsection{A conditional energy functional}\label{sec:cond_energy}
Following \cite[Lemma~5.1]{LiuLargetimeBehaviorTwodimensional2021}
(cf.\ also \cite[Lemma~3.1]{WinklerSmallmassSolutionsTwodimensional2020} and \cite[Lemma~3.7]{HeihoffTwoNewFunctional2022} for precedents for systems with linear taxis sensitivity),
we now consider the functional
\begin{align}\label{eq:def_f}
  \mc F_{K, L, M, \eps} \defs 
    \intom \neps \ln \left( \frac{\neps}{\ol n_0} \right)
  + \frac{K}{2} \intom |\nabla \we|^2
  + \frac1{2L} \intom |\ue|^2
  + \frac{M}{2} \intom \ce^2
\end{align}
for $K, L, M > 0$ and $\eps \in (0, 1)$.
We crucially rely on \eqref{eq:ev_smooth:cond_f} in order to show that $\mc F_{K, L, M, \eps}$ is a conditional energy functional for suitably chosen parameters.
This is an important difference to \cite[Lemma~5.1]{LiuLargetimeBehaviorTwodimensional2021}, where a similar result (for $(K, L, M) = (1, 1, 0)$) has been shown under the assumption that $\intom n_0$ is sufficiently small.
\begin{lemma}\label{lm:cond_func}
  There exist $K, L, M > 0$, $\eta_0 > 0$ and $\newgc{ev_space_time_bdd} > 0$ with the following property:
  Let $\eta \in (0, \eta_0)$, $\eps \in (0, 1)$ and $\mc F_{K, L, M, \eps}$ be as in \eqref{eq:def_f}.
  If there is $t_0 \ge 0$ such that
  \begin{align}\label{eq:cond_func:cond}
    \mc F_{K, L, M, \eps}(t_0) < \eta,
  \end{align}
  then
  \begin{align}\label{eq:cond_func:statement1}
    \mc F_{K, L, M, \eps}(t) \le \eta
    \qquad \text{for all $t \ge t_0$}
  \end{align}
  and
  \begin{align}\label{eq:cond_func:statement2}
      \int_{t_0}^\infty \intom \frac{|\nabla \neps|^2}{\neps}
    + \int_{t_0}^\infty \intom |\Delta \we|^2
    + \int_{t_0}^\infty \intom |\nabla \ue|^2 
    + \int_{t_0}^\infty \intom |\nabla \ce|^2 
    \le \gc{ev_space_time_bdd}.
  \end{align}
\end{lemma}
\begin{proof}
  We begin by fixing several constants.
  By the Poincar\'e inequality, there is $\newlc{poincare} > 0$ such that
  \begin{align}\label{eq:cond_func:poincare1}
    \intom |\varphi|^2 \le \lc{poincare} \intom |\nabla \varphi|^2
    \qquad \text{for all $\varphi \in C^1(\Ombar; \R^2)$ with $\varphi = 0$ on $\partial \Omega$}
  \end{align}
  as well as
  \begin{align}\label{eq:cond_func:poincare2}
    \intom |\nabla \varphi|^2 \le \lc{poincare} \intom |\Delta \varphi|^2
    \qquad \text{for all $\varphi \in \con2$ with $\partial_\nu \varphi = 0$ on $\partial \Omega$}
  \end{align}
  and since additionally $\sob11 \embed \leb2$, there is $\newlc{sobolev} > 0$ with
  \begin{align}\label{eq:cond_func:sobolev}
        \intom (\varphi - \ol \varphi)^2
    \le \lc{sobolev} \left( \intom |\nabla \varphi| \right)^2
    \le \lc{sobolev} \left( \intom \varphi \right) \intom \frac{|\nabla \varphi|^2}{\varphi}
    \qquad \text{for all $0 < \varphi \in \con1$}.
  \end{align}
  Moreover, by the Gagliardo--Nirenberg and the Poincar\'e inequalities, we can find $\newlc{gni} > 0$ such that
  \begin{align}\label{eq:cond_func:gni}
    \intom |\nabla \varphi|^4 \le \lc{gni} \left( \intom |\Delta \varphi|^2 \right) \left( \intom |\nabla \varphi|^2 \right)
    \qquad \text{for all $\varphi \in \con2$ with $\partial_\nu \varphi = 0$ on $\partial \Omega$}.
  \end{align}
  Next, we set
  \begin{align}\label{eq:cond_func:def_Kg}
    K \defs 16 \lc{poincare} \chi^2 \ol n_0,
    \quad \text{and} \quad
    g(s) \defs K [s^{-1} f(s)]^2 + s^{-1} f(s) + |f'(s)| \text{ for $s \ge 0$},
  \end{align}
  abbreviate
  \begin{align}\label{eq:cond_func:g_max}
    \newlc{g_max} \defs \max_{s \in [0, \|c_0\|_{\leb\infty}]} g(s)
  \end{align}
  (which is finite by \eqref{eq:intro:cond_f})
  and choose $A > 0$ so small that
  \begin{align}\label{eq:cond_func:g_small}
    g(s) \le \frac{1}{16 \lc{poincare} \ol n_0} \qquad \text{for $s \le A$}
  \end{align}
  which is possible due to \eqref{eq:intro:cond_f} and \eqref{eq:ev_smooth:cond_f}.
  Finally, we set
  \begin{align}\label{eq:cond_func:def_LM}
    L \defs 2\lc{poincare} \lc{sobolev} \|\nabla\phi\|_{\leb\infty}^2 \|n_0\|_{\leb1} + 1
    \quad \text{and} \quad
    M \defs 2A^{-2} K \lc{g_max} \ol n_0,
  \end{align}
  abbreviate
  \begin{align}\label{eq:cond_func:def_nabla_w_l4_factor}
    \newlc{nabla_w_l4_factor} \defs K \left( KL+ \frac12 \right), \quad
    \newlc{nabla_w_l4_factor2} \defs 2\lc{sobolev} |\Omega| \ol n_0 (K \lc{g_max} + \chi^2)^2
    \quad \text{and} \quad
    \newlc{f_factor} \defs \frac{2\lc{gni} (\lc{nabla_w_l4_factor}+\lc{nabla_w_l4_factor2})}{K},
  \end{align}
  and fix
  \begin{align}\label{eq:cond_fund:def_eta}
    \eta_0 \defs \frac{K}{8 \lc{f_factor}}
  \end{align}
  as well as $\eta \in (0, \eta_0)$ and $\eps \in (0, 1)$.
  With these preparations at hand, we are now able to estimate the time derivative of each summand appearing in \eqref{eq:def_f}.
  Starting with the contributions from the first and third solution components,
  we integrate by parts and apply Young's inequality as well as \eqref{eq:cond_func:poincare1}, \eqref{eq:cond_func:sobolev} and \eqref{eq:cond_func:def_LM} to obtain
  \begin{align}\label{eq:cond_func:n_part}
          \ddt \intom \neps \ln \left( \frac{\neps}{\ol n_0} \right)
    &=    - \intom (\nabla \neps + \chi \frac{\neps}{1+\eps\neps}  \nabla \we ) \cdot \nabla \ln \neps \notag \\
    &=    - \intom \frac{|\nabla \neps|^2}{\neps}
          - \chi \intom \frac{1}{1+\eps\neps} \nabla \neps \cdot \nabla \we \notag \\
    &\le  - \frac34 \intom \frac{|\nabla \neps|^2}{\neps}
          + \chi^2 \intom \neps |\nabla \we|^2
  \end{align}
  and
  \begin{align}\label{eq:cond_func:u_part}
          \frac12 \ddt \intom |\ue|^2
    &=    - \intom |\nabla \ue|^2
          + \intom (\neps - \ol n_0) (\nabla \phi \cdot \ue) \notag \\
    &\le  - \intom |\nabla \ue|^2
          + \frac{\lc{poincare}\|\nabla \phi\|_{\leb\infty}^2 }{2} \intom |\neps - \ol n_0|^2
          + \frac{1}{2\lc{poincare}} \intom |\ue|^2 \notag \\
    &\le  - \frac12 \intom |\nabla \ue|^2
          + \frac{\lc{poincare} \lc{sobolev} \|\nabla \phi\|_{\leb\infty}^2}{2} \left( \intom n_0 \right) \intom \frac{|\nabla \neps|^2}{\neps} \notag \\
    &\le  - \frac12 \intom |\nabla \ue|^2
          + \frac{L}{4} \intom \frac{|\nabla \neps|^2}{\neps}
  \end{align}
  in $(0, \infty)$.
  Moreover, by testing the second solution component with $\ce$ and the last equation in \eqref{eq:w_eg} with $-\Delta \we$, integrating by parts and applying Young's inequality, we see that
  \begin{align}\label{eq:cond_func:c_part}
          \frac12 \ddt \intom \ce^2 
    &=    - \intom |\nabla \ce|^2 - \intom \neps f(\ce) \ce
    \le   - \intom |\nabla \ce|^2
  \end{align}
  and
  \begin{align}\label{eq:cond_func:w_part}
    &\pe  \frac12 \ddt \intom |\nabla \we|^2 \notag \\
    &=    - \intom |\Delta \we|^2
          - \intom \nabla (\ue \cdot \nabla \we) \cdot \nabla \we
          + \intom |\nabla \we|^2 \Delta \we
          + \intom \nabla (\neps \ce^{-1} f(\ce)) \cdot \nabla \we \notag \\
    &=    - \intom |\Delta \we|^2
          - \intom (\nabla \ue \nabla \we) \cdot \nabla \we
          - \frac12 \intom \ue \cdot \nabla |\nabla \we|^2
          + \intom |\nabla \we|^2 \Delta \we
          + \intom \nabla (\neps \ce^{-1} f(\ce)) \cdot \nabla \we \notag \\
    &\le  - \frac12 \intom |\Delta \we|^2
          + \frac{1}{4KL}\intom |\nabla \ue|^2
          + \left( KL+ \frac12 \right) \intom |\nabla \we|^4
          + \intom \nabla (\neps \ce^{-1} f(\ce)) \cdot \nabla \we
  \end{align}
  in $(0, \infty)$ since $\nabla \cdot \ue = 0$ in $\Omega \times (0, \infty)$.
  Regarding the last term on the right-hand side in \eqref{eq:cond_func:w_part}, we again rely on Young's inequality and recall \eqref{eq:cond_func:def_Kg}, which entails the definition of $g$, in order to estimate
  \begin{align}\label{eq:cond_func:w_part2}
          \intom \nabla (\neps \ce^{-1} f(\ce)) \cdot \nabla \we
    &=    \intom \ce^{-1} f(\ce) \nabla \neps \cdot \nabla \we
          + \intom \neps (-\ce^{-1} f(\ce) + f'(\ce)) \frac{\nabla \ce}{\ce} \cdot \nabla \we \notag \\
    &\le  \frac1{4K} \intom \frac{|\nabla \neps|^2}{\neps}
          + \intom \neps g(\ce) |\nabla \we|^2
    \qquad \text{in $(0, \infty)$}.
  \end{align}
  Thus, combining \eqref{eq:cond_func:n_part}--\eqref{eq:cond_func:w_part2} and recalling that $\lc{nabla_w_l4_factor} = K \left( KL+ \frac12 \right)$ by \eqref{eq:cond_func:def_nabla_w_l4_factor} yields
  \begin{align}\label{eq:cond_func:ddt_f_1}
    &\pe  \ddt \mc F_{K, L, M, \eps}
          + \frac14 \intom \frac{|\nabla \neps|^2}{\neps}
          + \frac{K}{2} \intom |\Delta \we|^2
          + \frac1{4L} \intom |\nabla \ue|^2
          + M \intom |\nabla \ce|^2 \notag \\
    &\le  \lc{nabla_w_l4_factor} \intom |\nabla \we|^4
          + \intom \neps \left(K g(\ce) + \chi^2\right) |\nabla \we|^2
    \qquad \text{in $(0, \infty)$}.
  \end{align}
  As to the last term on the right-hand side in \eqref{eq:cond_func:ddt_f_1},
  we make use of \eqref{eq:cond_func:g_max}, Young's inequality, \eqref{eq:cond_func:def_nabla_w_l4_factor}, \eqref{eq:cond_func:def_Kg}, \eqref{eq:cond_func:g_small}, \eqref{eq:cond_func:sobolev},
  \eqref{eq:cond_func:poincare2} and \eqref{eq:cond_func:def_LM} in estimating
  \begin{align}\label{eq:cond_func:est_last_rhs}
    &\pe  \intom \neps \left(K g(\ce) + \chi^2\right) |\nabla \we|^2 \notag \\
    &\le  (K \lc{g_max} + \chi^2) \intom |\neps - \ol n_0| |\nabla \we|^2
          + \chi^2 \ol n_0 \intom |\nabla \we|^2
          + K \ol n_0 \int_{\{\ce \le A\}} g(\ce) |\nabla \we|^2
          + K \ol n_0 \int_{\{\ce > A\}} g(\ce) |\nabla \we|^2 \notag \\
    &\le  \frac{1}{8\lc{sobolev} |\Omega| \ol n_0} \intom (\neps - \ol n_0)^2
          + \lc{nabla_w_l4_factor2} \intom |\nabla \we|^4
          + \left( \frac{K}{16 \lc{poincare}} + \frac{K}{16 \lc{poincare}} \right) \intom |\nabla \we|^2
          + A^{-2} K \lc{g_max} \ol n_0 \intom |\nabla \ce|^2 \notag \\
    &\le  \frac18 \intom \frac{|\nabla \neps|^2}{\neps}
          + \lc{nabla_w_l4_factor2} \intom |\nabla \we|^4
          + \frac{K}{8} \intom |\Delta \we|^2
          + \frac{M}{2} \intom |\nabla \ce|^2
    \qquad \text{in $(0, \infty)$.}
  \end{align}
  Since moreover
  \begin{align*}
          \intom |\nabla \we|^4
    &\le  \lc{gni} \left( \intom |\Delta \we|^2 \right) \left( \intom |\nabla \we|^2 \right)
     \le  \frac{2\lc{gni}}{K} \left( \intom |\Delta \we|^2 \right) \mc F_{K, L, M, \eps}
    \qquad \text{in $(0, \infty)$ for all $\eps \in (0, 1)$}
  \end{align*}
  by \eqref{eq:cond_func:gni} and \eqref{eq:def_f},
  we conclude from \eqref{eq:cond_func:ddt_f_1} and \eqref{eq:cond_func:est_last_rhs} that
  \begin{align}\label{eq:cond_func:ddt_f_2}
    &\pe  \ddt \mc F_{K, L, M, \eps}
          + \frac18 \intom \frac{|\nabla \neps|^2}{\neps}
          + \frac{K}{4} \intom |\Delta \we|^2
          + \frac1{4L} \intom |\nabla \ue|^2
          + \frac{M}{2} \intom |\nabla \ce|^2 \notag \\
    &\le  \left( \lc{f_factor} \mc F_{K, L, M, \eps} - \frac{K}{8} \right) \intom |\Delta \we|^2
    \qquad \text{in $(0, \infty)$}.
  \end{align}
  We now assume that \eqref{eq:cond_func:cond} holds for some $t_0 > 0$ and claim that then \eqref{eq:cond_func:ddt_f_2} implies \eqref{eq:cond_func:statement1}.
  Indeed, suppose that $S \defs \{\,t \in (t_0, \infty) : \mc F_{K, L, M, \eps} \le \eta \text{ in } (0,t)\,\}$ does not coincide with $(t_0, \infty)$.
  As \eqref{eq:cond_func:cond} and continuity of $\mc F_{K, L, M, \eps}$ imply that $S$ is not empty,
  $t_\star \defs \sup S \in (t_0, \infty)$ is well-defined and satisfies $\mc F_{K, L, M, \eps}(t_\star) = \eta$ and $\mc F_{K, L, M, \eps} < \eta$ in $(0, t_\star)$.
  Thus, \eqref{eq:cond_fund:def_eta} entails that $\lc{f_factor} \mc F_{K, L, M, \eps} - \frac{K}{8} \le 0$ in $(t_0, t_\star)$
  so that integrating \eqref{eq:cond_func:ddt_f_2} and making use of \eqref{eq:cond_func:cond} yields
  \begin{align*}
        \eta
    =   \mc F_{K, L, M, \eps}(t_\star)
    \le \mc F_{K, L, M, \eps}(t_0)
    <   \eta,
  \end{align*}
  a contradiction. Thus \eqref{eq:cond_func:statement1} holds; that is, $\lc{f_factor} \mc F_{K, L, M, \eps} - \frac{K}{8} \le 0$ in $(t_0, \infty)$.
  By again integrating \eqref{eq:cond_func:ddt_f_2} and applying \eqref{eq:cond_func:cond},
  we finally obtain \eqref{eq:cond_func:statement2} for $\gc{ev_space_time_bdd} \defs \eta \max\{8, \frac4K, 4L, \frac 2M\}$.
\end{proof}

\subsection{Eventual smallness of the conditional energy functional}\label{sec:ev_small}
The goal of this subsection is to show that \eqref{eq:cond_func:cond} holds for some $t_0 > 0$, i.e.\ that we may apply Lemma~\ref{lm:cond_func}.
To that end, we first show in the following lemma that the first two summands in \eqref{eq:def_f} become small for a sequence of times going to infinity.
Its proof mainly rests on \eqref{eq:quasi_energy:ddt} which rapidly yields the desired estimate provided $\intom n_0$ is sufficiently small (cf.\ \cite[Lemma~5.2]{LiuLargetimeBehaviorTwodimensional2021}).
However, that argument appears to be insufficient for large initial mass and so we make again use of the key assumption \eqref{eq:ev_smooth:cond_f},
which allows us to favourable split the right-hand side in \eqref{eq:quasi_energy:ddt} into integrals over subdomains where $\ce$ is small and large, respectively.
\begin{lemma}\label{lm:nabla_ln_n_nabla_w_uniform_spacetime}
  Let $\eta > 0$ and $t_1 \ge 0$. Then there exists $\newgc{nabla_ln_n_nablaw_w_uniform_spacetime} > 0$ such that
  \begin{align}\label{eq:nabla_ln_n_nabla_w_uniform_spacetime:statement}
    \frac1{T-t_1} \int_{t_1}^T \intom \frac{|\nabla \neps|^2}{\neps^2} + \frac1{T-t_1} \int_{t_1}^T \intom |\nabla \we|^2 + \frac1{T-t_1} \int_{t_1}^T \intom \neps f(\ce)
    &\le \eta + \frac{\gc{nabla_ln_n_nablaw_w_uniform_spacetime}}{T-t_1}
  \end{align}
  for all $T > t_1$ and all $\eps \in (0, 1)$.
\end{lemma}
\begin{proof}
  We first make use of the key assumption \eqref{eq:ev_smooth:cond_f} to find $A > 0$ such that
  \begin{align*}
     \frac{f(s)}{s} \le \frac{\eta \min\{1, \chi^2\}}{2\chi^2 \intom n_0} \qquad \text{for all $s \in (0, A)$}.
  \end{align*}
  Noting that
  \begin{align*}
    \ddt \intom \ce = - \intom \neps f(\ce)
    \qquad \text{in $(0, \infty)$ for all $\eps \in (0, 1)$},
  \end{align*}
  and recalling \eqref{eq:quasi_energy:ddt}, we find that
  \begin{align*}
    y_\eps(t) \defs - \intom \ln \neps(\cdot, t) + \chi^2 \intom \we(\cdot, t) + \left(\frac{\chi^2}{A} + \frac12 \right) \intom \ce(\cdot, t), \quad t \in (0, \infty), \eps \in (0, 1)
  \end{align*}
  fulfills
  \begin{align*}
          \ddt y_\eps(t)
          + \frac12 \intom \frac{|\nabla \neps|^2}{\neps^2} 
          + \frac{\chi^2}2 \intom |\nabla \we|^2
          + \left(\frac{\chi^2}{A} + \frac12 \right) \intom \neps f(\ce)
    &\le  \chi^2 \intom \neps \ce^{-1} f(\ce)
  \end{align*}
  in $(0, \infty)$ for all $\eps \in (0, 1)$.
  Since
  \begin{align*}
        \intom \neps \ce^{-1} f(\ce)
    &\le \frac{\eta \min\{1, \chi^2\}}{2\chi^2  \intom n_0} \int_{\{\ce < A\}} \neps
        + \frac{1}{A} \int_{\{\ce \ge A\}} \neps f(\ce) \\
    &\le \frac{\eta \min\{1, \chi^2\}}{2\chi^2}
        + \frac{1}{A} \intom \neps f(\ce)
    \qquad \text{in $(0, \infty)$ for all $\eps \in (0, 1)$},
  \end{align*}
  we conclude that
  \begin{align}\label{eq:nabla_ln_n_nabla_w_uniform_spacetime:func_le_eta}
    &\pe  \ddt y_\eps(t)
          + \frac12 \intom \frac{|\nabla \neps|^2}{\neps^2} 
          + \frac{\chi^2}{2} \intom |\nabla \we|^2
          + \frac12 \intom \neps f(\ce)
     \le  \frac{\eta \min\{1, \chi^2\}}{2}
    \qquad \text{in $(0, \infty)$ for all $\eps \in (0, 1)$}.
  \end{align}
  Integrating this inequality first over $(0, t_1)$ and then over $(t_1, T)$ gives $y_\eps(t_1) \le y_\eps(0) + \frac{t_1 \eta \min\{1, \chi^2\}}{2}$
  and
  \begin{align*}
          \int_{t_1}^T \intom \frac{|\nabla \neps|^2}{\neps^2} + \chi^2 \int_{t_1}^T \intom |\nabla \we|^2 + \int_{t_1}^T \intom \neps f(\ce)
    &\le  (T-t_1)\eta\min\{1, \chi^2\} + 2 y_\eps(t_1) \\
    &\le  (T-t_1)\eta\min\{1, \chi^2\} + 2 y_\eps(0) + t_1 \eta\min\{1, \chi^2\}
  \end{align*}
  for all $T > 0$ and $\eps \in (0, 1)$.
  Since $(y_\eps(0))_{\eps \in (0, 1)}$ is bounded due to \eqref{eq:init:def_nc} and \eqref{eq:init_approx:n}, dividing this estimate by $T-t_1$
  yields \eqref{eq:nabla_ln_n_nabla_w_uniform_spacetime:statement} for some $\gc{nabla_ln_n_nablaw_w_uniform_spacetime} > 0$.
\end{proof}

Next, we deal with the third summand in \eqref{eq:def_f}, again by making use of estimates obtained in Section~\ref{sec:global_ex} (namely, of \eqref{eq:fluid_energy:ddt}).
\begin{lemma}\label{lm:u_space_time}
  Let $t_2 > 0$. Then there exists $\newgc{u_space_time} > 0$ such that
  \begin{align*}
    \int_{t_2}^T \intom |\ue|^2 \le \gc{u_space_time} \left(1 + \int_{t_2}^T \intom \frac{|\nabla \neps|^2}{\neps^2} \right)
    \qquad \text{for all $T > 0$ and all $\eps \in (0, 1)$}.
  \end{align*}
\end{lemma}
\begin{proof}
  Making use of \eqref{eq:fluid_energy:ddt} twice and recalling \eqref{eq:quasi_energy:integrated},
  we obtain (with $\gc{quasi_energy}$ and $\gc{fluid_n_ln_n}$ as in Lemma~\ref{lm:quasi_energy} and Lemma~\ref{lm:fluid_energy}, respectively)
  \begin{align*}
        \int_{t_2}^T \intom |\nabla \ue|^2
    \le \intom |\ue(\cdot, t_2)|^2 + \gc{fluid_n_ln_n} \int_{t_2}^T \intom \frac{|\nabla \neps|^2}{\neps^2}
    \le \intom |\une|^2 + \gc{quasi_energy} \gc{fluid_n_ln_n} (t_2+1) + \gc{fluid_n_ln_n} \int_{t_2}^T \intom \frac{|\nabla \neps|^2}{\neps^2}
  \end{align*}
  for all $T > 0$ and all $\eps \in (0, 1)$,
  upon which the statement follows due to the Poincar\'e inequality and by noting that $(\une)_{\eps \in (0, 1)}$ is bounded in $\leb2$ by \eqref{eq:init:def_u}.
\end{proof}

The last preparatory step before being able to apply Lemma~\ref{lm:cond_func} is to make sure that the last summand in \eqref{eq:def_f} also becomes small,
which due to \eqref{eq:basic_apriori} comes down to showing eventual smallness of $\intom \ce$.
The following proof is quite similar to \cite[Lemma~3.5]{HeihoffTwoNewFunctional2022}
and makes essential use of the Heihoff inequalities \eqref{eq:heihoff_ineq:1} and \eqref{eq:heihoff_ineq:2} derived in \cite{HeihoffTwoNewFunctional2022}.
\begin{lemma}\label{lm:c_l1_small}
  Let $\eta > 0$. Then there is $t_3 > 0$ such that 
  \begin{align*}
    \intom \ce(\cdot, t) \le \eta
    \qquad \text{for all $t \ge t_3$ and all $\eps \in (0, 1)$}.
  \end{align*}
\end{lemma}
\begin{proof}
  According to \eqref{eq:heihoff_ineq:1} (with $\eta=2$), \eqref{eq:heihoff_ineq:2} and \eqref{eq:basic_apriori}, and with $\gc{heihoff}$ as in Lemma~\ref{lm:heihoff_ineq},
  \begin{align}\label{eq:c_l1_small:f_ce_est}
          \intom f(\ce)
    &=    \frac1{\ol n_0} \intom (\neps - \ol n_0) f(\ce)
          + \frac1{\ol n_0} \intom \neps f(\ce) \notag \\
    &\le  \frac{|\Omega|}{2 \gc{heihoff}} \intom \frac{|\nabla \neps|^2}{\neps^2}
          + \frac{|\Omega|}{2 \gc{heihoff}} \intom |f'(\ce)|^2 |\nabla \ce|^2
          + \frac1{\ol n_0} \intom \neps f(\ce) \notag \\
    &\le  \lc{factor} \left(
            \intom \frac{|\nabla \neps|^2}{\neps^2}
            + \intom |\nabla \we|^2
            + \intom \neps f(\ce)
          \right)
  \end{align}
  holds in $(0, \infty)$ for all $\eps \in (0, 1)$, 
  where $\newlc{factor} \defs \max\{\frac{|\Omega|}{2\gc{heihoff}}, \frac{|\Omega|}{2\gc{heihoff}}\|f'\|_{L^\infty((0, \|c_0\|_{\leb\infty}))}^2 \|c_0\|_{\leb\infty}^2, \frac1{\ol n_0}, 1\}$.
  Next, we set
  \begin{align*}
    A \defs \frac{\eta}{2|\Omega|}, \quad
    \newlc{inf_f} \defs \min_{s \in [A, \|c_0\|_{\leb\infty}]} f(s) > 0
    \quad \text{and} \quad
    \tilde \eta \defs \frac{\lc{inf_f} \eta}{\lc{factor}4\|c_0\|_{\leb\infty}}
  \end{align*}
  and infer from \eqref{eq:c_l1_small:f_ce_est} and Lemma~\ref{lm:nabla_ln_n_nabla_w_uniform_spacetime} (applied to $\tilde \eta$ and $t_1 \defs 0$) that
  \begin{align*}
    \frac1T \int_0^T \intom f(\ce) \le \tilde \eta + \frac{\gc{nabla_ln_n_nablaw_w_uniform_spacetime}\lc{factor}}{T}
    \qquad \text{for all $T > 0$ and all $\eps \in (0, 1)$},
  \end{align*}
  where $\gc{nabla_ln_n_nablaw_w_uniform_spacetime}$ is given by Lemma~\ref{lm:nabla_ln_n_nabla_w_uniform_spacetime}.
  For $T = t_3 \defs \frac{\gc{nabla_ln_n_nablaw_w_uniform_spacetime}\lc{factor}}{\tilde \eta} > 0$, this in particular entails
  \begin{align*}
    \frac1{t_3} \int_0^{t_3} \intom f(\ce) \le 2\tilde \eta
    \qquad \text{for all $\eps \in (0, 1)$},
  \end{align*}
  which implies that for each $\eps \in (0, 1)$, there is $t_\eps \in (0, t_3)$ with
  \begin{align*}
    \intom f(\ce(\cdot, t_\eps)) \le 2\tilde \eta \le \frac{\lc{inf_f} \eta}{2\|c_0\|_{\leb\infty}}.
  \end{align*}
  As $\intom \ce$ is decreasing in time due to $\ddt \intom \ce = -\intom \neps f(\ce) \le 0$ in $(0, \infty)$ for all $\eps \in (0, 1)$, we thus conclude
  \begin{align*}
         \intom \ce(\cdot, t)
    &\le \intom \ce(\cdot, t_\eps)
    =    \int_{\{\ce < A\}} \ce(\cdot, t_\eps) + \int_{\{\ce \ge A\}} \ce(\cdot, t_\eps)
    \le A |\Omega| + \frac{\|c_0\|_{\leb\infty}}{\lc{inf_f}} \intom f(\ce(\cdot, t_\eps))
    \le  \frac{\eta}{2} + \frac{\eta}{2}
    =    \eta
  \end{align*}
  for all $t \ge t_3 \ge t_\eps$ and all $\eps \in (0, 1)$, as desired.
\end{proof}

The above lemmata now indeed render Lemma~\ref{lm:cond_func} applicable, allowing us to obtain several useful \emph{eventual} a~priori estimates.
\begin{lemma}\label{lm:cond_func_result}
  Let $\eta > 0$. Then there exist $t_4 > 0$ and $\newgc{ev_space_time_bdd2} > 0$ such that
  \begin{align}\label{eq:cond_func_result:statement1}
    \mc F_{K, L, M, \eps}
    \le \eta
    \qquad \text{in $(t_4, \infty)$ for all $\eps \in (0, 1)$}
  \end{align}
  and
  \begin{align}\label{eq:cond_func_result:statement2}
      \int_{t_4}^\infty \intom \frac{|\nabla \neps|^2}{\neps}
    + \int_{t_4}^\infty \intom |\Delta \we|^2
    + \int_{t_4}^\infty \intom |\nabla \ue|^2 
    + \int_{t_4}^\infty \intom |\nabla \ce|^2
    \le \gc{ev_space_time_bdd2}
    \qquad \text{for all $\eps \in (0, 1)$}.
  \end{align}
\end{lemma}
\begin{proof}
  We let $K, L, M > 0$, $\eta_0 > 0$ and $\gc{ev_space_time_bdd} > 0$ be as given by Lemma~\ref{lm:cond_func} and assume without loss of generality that $\eta < \eta_0$.
  By Lemma~\ref{lm:c_l1_small} and \eqref{eq:basic_apriori}, there is $t_3 > 0$ such that
  \begin{align}\label{eq:cond_func_result_c_l2_small}
    \frac{M}{2} \intom \ce^2(\cdot, t) \le \frac{\eta}{4}
    \qquad \text{for all $t \ge t_3$}.
  \end{align}
  Next, we set
  \begin{align*}
    \tilde \eta \defs \frac{\eta}{8}\left(\max\left\{\frac{2 \intom n_0}{\gc{heihoff}}, K \gc{u_space_time}, \frac{1}{2L}\right\}\right)^{-1},
  \end{align*}
  where $\gc{heihoff}$ and $\gc{u_space_time}$ are given by Lemma~\ref{lm:heihoff_ineq} and Lemma~\ref{lm:u_space_time} (applied to $t_2 = t_3$),
  and apply Lemma~\ref{lm:nabla_ln_n_nabla_w_uniform_spacetime} to $\tilde \eta$ and $t_1 = t_3$ to obtain $\gc{nabla_ln_n_nablaw_w_uniform_spacetime} > 0$ such that
  \begin{align*}
    \frac1{T-t_3} \int_{t_3}^T \intom \frac{|\nabla \neps|^2}{\neps^2} + \frac1{T-t_3} \int_{t_3}^T \intom |\nabla \we|^2 \le \tilde \eta + \frac{\gc{nabla_ln_n_nablaw_w_uniform_spacetime}}{T-t_3}
    \qquad \text{for all $T > t_3$ and $\eps \in (0, 1)$}.
  \end{align*}
  For $T = t_4 \defs t_3 + \frac{\gc{nabla_ln_n_nablaw_w_uniform_spacetime} + \frac12}{\tilde \eta}$,
  in conjunction with Lemma~\ref{lm:heihoff_ineq} and Lemma~\ref{lm:u_space_time} this further implies
  \begin{align*}
        \frac1{t_4-t_3} \int_{t_3}^{t_4}\left(
          \frac{\gc{heihoff}}{2 \intom n_0} \intom \neps \ln \frac{\neps}{\ol n_0}
          + \frac{1}{2\gc{u_space_time}} \intom |\ue|^2
          + \intom |\nabla \we|^2
        \right)
    \le \tilde \eta + \frac{\gc{nabla_ln_n_nablaw_w_uniform_spacetime} + \frac12}{t_4-t_3}
    =   2\tilde \eta
    \qquad \text{for all $\eps \in (0, 1)$}
  \end{align*}
  and thus
  \begin{align*}
    &\pe \frac1{t_4-t_3} \int_{t_3}^{t_4} \left( \intom \neps \ln \left( \frac{\neps}{\ol n_0} \right) + \frac{K}{2} \intom |\nabla \we|^2 + \frac1{2L} \intom |\ue|^2 \right) \\
    &\le 2 \tilde \eta \max\left\{\frac{2 \intom n_0}{\gc{heihoff}}, \frac{K}{2}, \frac{\gc{u_space_time}}{L}\right\} 
    =    \frac{\eta}{4}
    \qquad \text{for all $\eps \in (0, 1)$}.
  \end{align*}
  In combination with \eqref{eq:cond_func_result_c_l2_small}, this shows that for each $\eps \in (0, 1)$ there exists $t_\eps \in (t_3, t_4)$ with $\mc F_{K, L, M, \eps}(t_\eps) \le \frac{\eta}{2}$.
  Thus, we may apply Lemma~\ref{lm:cond_func} to $t_0 = t_\eps$
  and due to $t_\eps \le t_4$, the estimates \eqref{eq:cond_func:statement1} and \eqref{eq:cond_func:statement2} imply \eqref{eq:cond_func_result:statement1} and \eqref{eq:cond_func_result:statement2}
  for some $\gc{ev_space_time_bdd2} > 0$.
\end{proof}

\subsection{Higher order estimates and large-time behavior}\label{sec:ev_smooth_large_time}
The estimates \eqref{eq:cond_func_result:statement1} and \eqref{eq:cond_func_result:statement2} can be considerably strengthened by a well-established bootstrap procedure,
so that we can improve on the convergence properties asserted by Lemma~\ref{lm:eps_sea_0} and Lemma~\ref{lm:strong_conv}
and conclude that the solution given by Theorem~\ref{th:global_ex} eventually also solves \eqref{prob:fluid_incl} classically.
\begin{lemma}\label{lm:c2_bdd}
  Let $(n, c, u)$ and $(\eps_j)_{j \in \N}$ be given by Lemma~\ref{lm:eps_sea_0}.
  There exist $t_\star > 0$, $\alpha \in (0, 1)$, $\newgc{c2_bdd} > 0$ and a subsequence of $(\eps_j)_{j \in \N}$, which we do not relabel,
  such that $n, c, u \in C^{2+\alpha, \frac{2+\alpha}{2}}(\Ombar \times [t_\star, \infty))$ with
  \begin{align}\label{eq:c2_bdd:bdd}
          \|n\|_{C^{2+\alpha, \frac{2+\alpha}{2}}(\Ombar \times [t, t+1])}
        + \|c\|_{C^{2+\alpha, \frac{2+\alpha}{2}}(\Ombar \times [t, t+1])}
        + \|u\|_{C^{2+\alpha, \frac{2+\alpha}{2}}(\Ombar \times [t, t+1])}
    \le \gc{c2_bdd}
  \end{align}
  for all $t \ge t_\star$
  and that
  \begin{align}\label{eq:c2_bdd:conv}
    (\neps, \ce, \ue) \to (n, c, u)
    \qquad \text{in } \left(C_{\loc}^{2+\alpha, \frac{2+\alpha}{2}}(\Ombar \times [t_\star, \infty))\right)^{1+1+2}
    \text{ as $\eps = \eps_j \sea0$}.
  \end{align}
  Moreover, there is $P \in C^{1, 0}(\Ombar \times [t_\star, \infty))$ such that $(n, c, u, P)$ fulfills the first four lines in \eqref{prob:fluid_incl} pointwise.
\end{lemma}
\begin{proof}
  Starting from the bounds given by Lemma~\ref{lm:cond_func_result}, a bootstrap procedure eventually gives sufficiently strong $\eps$-uniform bounds.
  As this strategy has already been executed successfully in closely related settings, we choose to only briefly sketch the main steps.
  In the following list, all estimates are $\eps$-independent
  and ``eventual bound'' means a bound holding starting from some minimum time which in turn may increase from estimate to estimate.
  (This is unproblematic as there are only finitely many steps.)
  One can obtain
  \begin{enumerate}
    \item an eventual $L^\infty$-$L^2$ bound for $\neps$ by testing, cf.\ \cite[Lemma~3.3]{WinklerSmallmassSolutionsTwodimensional2020},
    \item an eventual $L^\infty$-$L^2$ bound for $\nabla \ue$ by testing, cf.\ \cite[Lemma~3.4]{WinklerSmallmassSolutionsTwodimensional2020},
    \item an eventual $L^\infty$-$L^2$ bound for $A^\beta \ue$, $\beta \in (\frac12, 1)$, by semigroup methods, cf.\ \cite[Lemma~3.5]{WinklerSmallmassSolutionsTwodimensional2020},
    \item an eventual $L^\infty$-$L^4$ bound for $\nabla \we$ by testing, cf.\ \cite[Lemma~5.5]{LiuLargetimeBehaviorTwodimensional2021},
    \item an eventual $L^\infty$-$L^\infty$ bound for $\neps$ by semigroup methods, cf.\ \cite[Lemma~4.4]{BlackEventualSmoothnessGeneralized2018},
    \item an eventual $C^{1+\alpha_1, \frac{1+\alpha_1}{2}}$ bound for $\ue$ (for some $\alpha_1 \in (0, 1)$) by semigroup methods, cf.\ \cite[Lemma~3.16 and Corollary~3.17]{HeihoffTwoNewFunctional2022}, 
    \item an eventual $C^{\alpha_2, \frac{\alpha_2}{2}}$ bound for $\neps$ (for some $\alpha_2 \in (0, 1)$) by applying results on Hölder estimates for scalar parabolic equations, cf.\ \cite[Lemma~3.18]{HeihoffTwoNewFunctional2022} and
    \item finally eventual  $C^{2+\alpha_3, 1+\frac{\alpha_3}{2}}$ bounds for all solution components (for some $\alpha_3 \in (0, 1)$) by Schauder estimates, cf.\ \cite[Lemma~4.7]{BlackEventualSmoothnessGeneralized2018}.
  \end{enumerate}
  Fixing $\alpha \in (0, \alpha_3)$ and setting $t_\star$ as the time for which the last point in this list holds,
  we may apply the Arzel\`a--Ascoli theorem to obtain \eqref{eq:c2_bdd:conv} and then also the bounds in \eqref{eq:c2_bdd:bdd}.
  Well-known arguments then show that $(n, c, u)$ together with some $P$ forms a classical solution of \eqref{prob:fluid_incl} in $\Ombar \times [t_\star, \infty)$, cf.\ Remark~\ref{rm:sol_concept}(i).
\end{proof}

Next, we show that the eventual smallness properties asserted in \eqref{eq:cond_func_result:statement1} 
together with convergence in $\leb1$ as $\eps = \eps_j \sea 0$ at each (large) time point entailed in \eqref{eq:c2_bdd:conv}
imply convergence in $\leb1$ for $t \to \infty$ of the solution constructed in Theorem~\ref{th:global_ex}.
\begin{lemma}\label{lm:large_time_l1}
  The triple $(n, c, u)$ given by Theorem~\ref{th:global_ex} fulfills
  \begin{align}\label{eq:large_time_l1:statement}
    \lim_{t \to \infty} \left( \|n(\cdot, t) - \ol n_0\|_{\leb1} + \|c(\cdot, t)\|_{\leb1} + \|u(\cdot, t)\|_{\leb1} \right) = 0.
  \end{align}
\end{lemma}
\begin{proof}
  Let $\eta > 0$. By virtue of \eqref{eq:cond_func_result:statement1}, Lemma~\ref{lm:csiszar-kullback} and Hölder's inequality,
  there is $t_4 > 0$ such that
  \begin{align*}
    \|\neps(\cdot, t) - \ol n_0\|_{\leb1} + \|\ce(\cdot, t)\|_{\leb1} + \|\ue(\cdot, t)\|_{\leb1} < \frac{\eta}{2}
    \qquad \text{for all $t \ge t_4$ and all $\eps \in (0, 1)$}.
  \end{align*}
  Moreover, Lemma~\ref{lm:c2_bdd} asserts that (with $t_\star$ as given by that lemma) for each $t \ge t_\star$ there exists $\eps(t) \in (0, 1)$ such that
  \begin{align*}
    \|n_{\eps(t)}(\cdot, t) - n(\cdot, t) \|_{\leb1} + \|c_{\eps(t)}(\cdot, t) - c(\cdot, t)\|_{\leb1} + \|u_{\eps(t)}(\cdot, t) - u(\cdot, t)\|_{\leb1} < \frac{\eta}{2}.
  \end{align*}
  Combining these estimates gives
  \begin{align*}
    \|n(\cdot, t) - \ol n_0\|_{\leb1} + \|c(\cdot, t)\|_{\leb1} + \|u(\cdot, t)\|_{\leb1} < \eta
    \qquad \text{for all $t > \max\{\hat t, t_\star\}$}
  \end{align*}
  and thus \eqref{eq:large_time_l1:statement}.
\end{proof}

By a typical compactness argument, we can improve \eqref{eq:large_time_l1:statement} to convergence in $\con2$.
\begin{lemma}\label{lm:large_time_c2}
  The solution $(n, c, u)$ of \eqref{prob:fluid_incl} given by Theorem~\ref{th:global_ex} fulfills \eqref{eq:ev_smooth:large_time}.
\end{lemma}
\begin{proof}
  Suppose that \eqref{eq:ev_smooth:large_time} is false, i.e.\ that there are $\newlc{nonconv} > 0$ and a sequence $(t_j)_{j \in \N} \subset (0, \infty)$ with $t_j \to \infty$ for $j \to \infty$ such that
  \begin{align}\label{eq:large_time_c2:contr}
    \|n(\cdot, t_j) - \ol n_0\|_{\con2} + \|c(\cdot, t_j)\|_{\con2} + \|u(\cdot, t_j)\|_{\con2} \ge \lc{nonconv}
    \qquad \text{for all $j \in \N$}.
  \end{align}
  Thanks to \eqref{eq:c2_bdd:bdd} and as $\con{2+\alpha}$ embeds compactly into $\con2$ for any $\alpha \in (0, 1)$,
  we can find $n_\infty \in \con2$, $c_\infty \in \con2$, $u_\infty \in \con2$ and a subsequence $(t_{j_k})_{k \in \N}$ of $(t_j)_{j \in \N}$ such that
  \begin{align}\label{eq:large_time_c2:conv}
    n(\cdot, t_{j_k}) \ra n_\infty \text{ in $\con2$}, \quad
    c(\cdot, t_{j_k}) \ra c_\infty \text{ in $\con2$} \quad \text{and} \quad
    u(\cdot, t_{j_k}) \ra u_\infty \text{ in $\con2$}
  \end{align}
  as $k \to \infty$. Since $\con2 \embed \leb1$, Lemma~\ref{lm:large_time_l1} asserts $(n_\infty, c_\infty, u_\infty) = (\ol n_0, 0, 0)$,
  meaning that \eqref{eq:large_time_c2:conv} contradicts \eqref{eq:large_time_c2:contr}.
\end{proof}

\subsection{Proofs of Theorem~\ref{th:ev_smooth} and Theorem~\ref{th:global_ex_classical}}\label{sec:proof_ev_smooth}
At last, we prove Theorem~\ref{th:ev_smooth} and Theorem~\ref{th:global_ex_classical}.
All assertions of the former have already been proven above.
\begin{proof}[Proof of Theorem~\ref{th:ev_smooth}]
  That the solution $(n, c, u)$ constructed in Theorem~\ref{th:global_ex} becomes eventually smooth in the sense of Theorem~\ref{th:ev_smooth} has been shown in Lemma~\ref{lm:c2_bdd} 
  and that \eqref{eq:ev_smooth:large_time} holds has been asserted in Lemma~\ref{lm:large_time_c2}.
\end{proof}

Moreover, mainly as a byproduct of Lemma~\ref{lm:cond_func} we also obtain Theorem~\ref{th:global_ex_classical}.
\begin{proof}[Proof of Theorem~\ref{th:global_ex_classical}]
  The existence of local classical solutions to \eqref{prob:fluid_incl} can be seen as in Lemma~\ref{lm:global_ex_approx}.
  Noting that the constants $K, L, M$ and $\eta_0$ given by Lemma~\ref{lm:cond_func} only depend on the initial data due to dependence on $\intom n_0$ and $\|c_0\|_{\leb\infty}$
  (cf.\ \eqref{eq:cond_func:def_Kg}, \eqref{eq:cond_func:def_LM} and \eqref{eq:cond_fund:def_eta}),
  we may choose $\eta$ (depending on $m_0$) sufficiently small so that a version of Lemma~\ref{lm:cond_func} for $\eps=0$ becomes applicable 
  and yields boundedness of the functional defined in \eqref{eq:sol_concept:n_ln_supersol} for all $t \ge 0$.
  A bootstrap argument as in Lemma~\ref{lm:c2_bdd} then shows that maximal classical solutions with initial data fulfilling the conditions of Theorem~\ref{th:global_ex_classical} are locally bounded in $(\con2)^{1+1+2}$
  and hence exist globally in time.
\end{proof}

\section*{Acknowledgments}\pdfbookmark{Acknowledgments}{sec:ack}
The author would like to thank the Max Planck Institute for Mathematics in the Sciences
for providing access to the article \cite{CsiszarInformationtypeMeasuresDifference1967}.

\pdfbookmark{References}{sec:ref}

\footnotesize
\begin{thebibliography}{10}
\setlength{\itemsep}{0.2pt}

\bibitem{BellomoEtAlMultiscaleDerivationNonlinear2016}
\textsc{Bellomo, N.}, \textsc{Bellouquid, A.}, and \textsc{Chouhad, N.}:
\newblock {\em From a multiscale derivation of nonlinear cross-diffusion models
  to {{Keller-Segel}} models in a {{Navier-Stokes}} fluid}.
\newblock Math. Models Methods Appl. Sci.,
  \href{https://doi.org/10.1142/S0218202516400078}{26(11):2041--2069}, 2016.
\newblock

\bibitem{BellomoEtAlChemotaxisCrossdiffusionModels2022}
\textsc{Bellomo, N.}, \textsc{Outada, N.}, \textsc{Soler, J.}, \textsc{Tao,
  Y.}, and \textsc{Winkler, M.}:
\newblock {\em Chemotaxis and cross-diffusion models in complex environments:
  {{Models}} and analytic problems toward a multiscale vision}.
\newblock Math. Models Methods Appl. Sci.,
  \href{https://doi.org/10.1142/S0218202522500166}{32(4):713--792}, 2022.
\newblock

\bibitem{BlackEventualSmoothnessGeneralized2018}
\textsc{Black, T.}:
\newblock {\em Eventual smoothness of generalized solutions to a singular
  chemotaxis-{{Stokes}} system in {{2D}}}.
\newblock J. Differential Equations,
  \href{https://doi.org/10.1016/j.jde.2018.04.035}{265(5):2296--2339}, 2018.
\newblock

\bibitem{CaoLankeitGlobalClassicalSmalldata2016}
\textsc{Cao, X.} and \textsc{Lankeit, J.}:
\newblock {\em Global classical small-data solutions for a three-dimensional
  chemotaxis {{Navier}}\textendash{{Stokes}} system involving matrix-valued
  sensitivities}.
\newblock Calc. Var. Partial Differ. Equ.,
  \href{https://doi.org/10.1007/s00526-016-1027-2}{55(4)}, 2016.
\newblock

\bibitem{ChaeEtAlGlobalExistenceTemporal2014}
\textsc{Chae, M.}, \textsc{Kang, K.}, and \textsc{Lee, J.}:
\newblock {\em Global existence and temporal decay in {{Keller-Segel}} models
  coupled to fluid equations}.
\newblock Comm. Partial Differential Equations,
  \href{https://doi.org/10.1080/03605302.2013.852224}{39(7):1205--1235}, 2014.
\newblock

\bibitem{CsiszarInformationtypeMeasuresDifference1967}
\textsc{Csisz{\'a}r, I.}:
\newblock {\em Information-type measures of difference of probability
  distributions and indirect observations}.
\newblock Studia Sci. Math. Hungar., 2:299--318, 1967.

\bibitem{DombrowskiEtAlSelfConcentrationLargeScaleCoherence2004}
\textsc{Dombrowski, C.}, \textsc{Cisneros, L.}, \textsc{Chatkaew, S.},
  \textsc{Goldstein, R.~E.}, and \textsc{Kessler, J.~O.}:
\newblock {\em Self-{{Concentration}} and {{Large-Scale Coherence}} in
  {{Bacterial Dynamics}}}.
\newblock Phys. Rev. Lett.,
  \href{https://doi.org/10.1103/PhysRevLett.93.098103}{93(9):098103}, 2004.
\newblock

\bibitem{DuanEtAlGlobalSolutionsCoupled2010}
\textsc{Duan, R.}, \textsc{Lorz, A.}, and \textsc{Markowich, P.}:
\newblock {\em Global solutions to the coupled chemotaxis-fluid equations}.
\newblock Comm. Partial Differential Equations,
  \href{https://doi.org/10.1080/03605302.2010.497199}{35(9):1635--1673}, 2010.
\newblock

\bibitem{FuestStrongConvergenceWeighted2022}
\textsc{Fuest, M.}:
\newblock {\em Strong convergence of weighted gradients in parabolic equations
  and applications to global generalized solvability of cross-diffusive
  systems}.
\newblock Preprint, \href {https://arxiv.org/abs/2202.00317}
  {{arXiv:2202.00317}}, 2022.

\bibitem{FujieSenbaGlobalExistenceBoundedness2015}
\textsc{Fujie, K.} and \textsc{Senba, T.}:
\newblock {\em Global existence and boundedness in a parabolic--elliptic
  {{Keller--Segel}} system with general sensitivity}.
\newblock Discrete Contin. Dyn. Syst. - Ser. B,
  \href{https://doi.org/10.3934/dcdsb.2016.21.81}{21(1):81--102}, 2015.
\newblock

\bibitem{GaldiIntroductionMathematicalTheory2011}
\textsc{Galdi, G.~P.}:
\newblock {\em An Introduction to the Mathematical Theory of the
  {{Navier-Stokes}} Equations: {{Steady-state}} Problems}.
\newblock Springer Monographs in Mathematics. {Springer, New York}, second
  edition, 2011.
\newblock

\bibitem{HeihoffGlobalMasspreservingSolutions2020}
\textsc{Heihoff, F.}:
\newblock {\em Global mass-preserving solutions for a two-dimensional
  chemotaxis system with rotational flux components coupled with a full
  {{Navier}}\textendash{{Stokes}} equation}.
\newblock Discrete Contin. Dyn. Syst. - B,
  \href{https://doi.org/10.3934/dcdsb.2020120}{25(12):4703--4719}, 2020.
\newblock

\bibitem{HeihoffTwoNewFunctional2022}
\textsc{Heihoff, F.}:
\newblock {\em Two new functional inequalities and their application to the
  eventual smoothness of solutions to a
  chemotaxis\textendash{{Navier}}\textendash{{Stokes}} system with rotational
  flux}.
\newblock Preprint, \href{https://arxiv.org/abs/2211.00624}{ arXiv:2211.00624},
  2022.

\bibitem{HerreroVelazquezBlowupMechanismChemotaxis1997}
\textsc{Herrero, M.~A.} and \textsc{Vel{\'a}zquez, J. J.~L.}:
\newblock {\em A blow-up mechanism for a chemotaxis model}.
\newblock Ann. Della Scuola Norm. Super. Pisa Cl. Sci. Ser. IV, 24(4):633--683
  (1998), 1997.

\bibitem{HillenPainterUserGuidePDE2009}
\textsc{Hillen, T.} and \textsc{Painter, K.~J.}:
\newblock {\em A user's guide to {{PDE}} models for chemotaxis}.
\newblock J. Math. Biol.,
  \href{https://doi.org/10.1007/s00285-008-0201-3}{58(1-2):183--217}, 2009.
\newblock

\bibitem{HombergEtAlExistenceGeneralizedSolutions2022}
\textsc{H{\"o}mberg, D.}, \textsc{Lasarzik, R.}, and \textsc{Plato, L.}:
\newblock {\em On the existence of generalized solutions to a spatio-temporal
  predator-prey system}.
\newblock \href{https://doi.org/10.20347/WIAS.PREPRINT.2925}{Preprint}, 2022.
\newblock

\bibitem{JiangEtAlGlobalExistenceAsymptotic2015}
\textsc{Jiang, J.}, \textsc{Wu, H.}, and \textsc{Zheng, S.}:
\newblock {\em Global existence and asymptotic behavior of solutions to a
  chemotaxis-fluid system on general bounded domains}.
\newblock Asymptot. Anal.,
  \href{https://doi.org/10.3233/asy-141276}{92(3-4):249--258}, 2015.
\newblock

\bibitem{KellerSegelTravelingBandsChemotactic1971}
\textsc{Keller, E.~F.} and \textsc{Segel, L.~A.}:
\newblock {\em Traveling bands of chemotactic bacteria: A theoretical
  analysis}.
\newblock J. Theor. Biol.,
  \href{https://doi.org/10.1016/0022-5193(71)90051-8}{30(2):235--248}, 1971.
\newblock

\bibitem{LankeitLankeitGlobalGeneralizedSolvability2019}
\textsc{Lankeit, E.} and \textsc{Lankeit, J.}:
\newblock {\em On the global generalized solvability of a chemotaxis model with
  signal absorption and logistic growth terms}.
\newblock Nonlinearity,
  \href{https://doi.org/10.1088/1361-6544/aaf8c0}{32(5):1569--1596}, 2019.
\newblock

\bibitem{LankeitLocallyBoundedGlobal2017}
\textsc{Lankeit, J.}:
\newblock {\em Locally bounded global solutions to a chemotaxis consumption
  model with singular sensitivity and nonlinear diffusion}.
\newblock J. Differ. Equ.,
  \href{https://doi.org/10.1016/j.jde.2016.12.007}{262(7):4052--4084}, 2017.
\newblock

\bibitem{LankeitViglialoroGlobalExistenceBoundedness2020}
\textsc{Lankeit, J.} and \textsc{Viglialoro, G.}:
\newblock {\em Global existence and boundedness of solutions to a
  chemotaxis-consumption model with singular sensitivity}.
\newblock Acta Appl. Math.,
  \href{https://doi.org/10.1007/s10440-019-00269-x}{167(1):75--97}, 2020.
\newblock

\bibitem{LankeitWinklerFacingLowRegularity2019}
\textsc{Lankeit, J.} and \textsc{Winkler, M.}:
\newblock {\em Facing low regularity in chemotaxis systems}.
\newblock Jahresber. Dtsch. Math.-Ver.,
  \href{https://doi.org/10.1365/s13291-019-00210-z}{122:35--64}, 2019.
\newblock

\bibitem{LiuGlobalClassicalSolution2018}
\textsc{Liu, D.}:
\newblock {\em Global classical solution to a chemotaxis consumption model with
  singular sensitivity}.
\newblock Nonlinear Anal. Real World Appl. Int. Multidiscip. J.,
  \href{https://doi.org/10.1016/j.nonrwa.2017.11.004}{41:497--508}, 2018.
\newblock

\bibitem{LiuLargetimeBehaviorTwodimensional2021}
\textsc{Liu, J.}:
\newblock {\em Large-time behavior in a two-dimensional logarithmic
  chemotaxis-{{Navier}}\textendash{{Stokes}} system with signal absorption}.
\newblock J. Evol. Equ., 2021.
\newblock

\bibitem{MilnorTopologyDifferentiableViewpoint1965}
\textsc{Milnor, J.~W.}:
\newblock {\em Topology from the Differentiable Viewpoint}.
\newblock {University Press of Virginia, Charlottesville, Va.}, 1965.
\newblock Based on notes by David W. Weaver.

\bibitem{MizoguchiWinklerBlowupTwodimensionalParabolic}
\textsc{Mizoguchi, N.} and \textsc{Winkler, M.}:
\newblock {\em Blow-up in the two-dimensional parabolic
  {{Keller}}\textendash{{Segel}} system}.
\newblock Preprint.

\bibitem{MoserSharpFormInequality1970}
\textsc{Moser, J.}:
\newblock {\em A sharp form of an inequality by {{N}}. {{Trudinger}}}.
\newblock Indiana Univ. Math. J.,
  \href{https://doi.org/10.1512/iumj.1971.20.20101}{20:1077--1092}, 1970.
\newblock

\bibitem{NagaiSenbaGlobalExistenceBlowup1998}
\textsc{Nagai, T.} and \textsc{Senba, T.}:
\newblock {\em Global existence and blow-up of radial solutions to a
  parabolic-elliptic system of chemotaxis}.
\newblock Adv. Math. Sci. Appl., 8(1):145--156, 1998.

\bibitem{RosenSteadystateDistributionBacteria1978}
\textsc{Rosen, G.}:
\newblock {\em Steady-state distribution of bacteria chemotactic toward
  oxygen}.
\newblock Bull. Math. Biol.,
  \href{https://doi.org/10.1007/BF02460738}{40(5):671--674}, 1978.
\newblock

\bibitem{TaoWinklerEventualSmoothnessStabilization2012}
\textsc{Tao, Y.} and \textsc{Winkler, M.}:
\newblock {\em Eventual smoothness and stabilization of large-data solutions in
  a three-dimensional chemotaxis system with consumption of chemoattractant}.
\newblock J. Differ. Equ.,
  \href{https://doi.org/10.1016/j.jde.2011.07.010}{252(3):2520--2543}, 2012.
\newblock

\bibitem{TrudingerImbeddingsOrliczSpaces1967}
\textsc{Trudinger, N.~S.}:
\newblock {\em On imbeddings into {{Orlicz}} spaces and some applications}.
\newblock J. Math. Mech.,
  \href{https://doi.org/10.1512/iumj.1968.17.17028}{17:473--483}, 1967.
\newblock

\bibitem{TuvalEtAlBacterialSwimmingOxygen2005}
\textsc{Tuval, I.}, \textsc{Cisneros, L.}, \textsc{Dombrowski, C.},
  \textsc{Wolgemuth, C.~W.}, \textsc{Kessler, J.~O.}, and \textsc{Goldstein,
  R.~E.}:
\newblock {\em Bacterial swimming and oxygen transport near contact lines}.
\newblock Proc. Natl. Acad. Sci. U.S.A.,
  \href{https://doi.org/10.1073/pnas.0406724102}{102(7):2277--2282}, 2005.
\newblock

\bibitem{WangGlobalLargedataGeneralized2016}
\textsc{Wang, Y.}:
\newblock {\em Global large-data generalized solutions in a two-dimensional
  chemotaxis-{{Stokes}} system with singular sensitivity}.
\newblock Bound. Value Probl.,
  \href{https://doi.org/10.1186/s13661-016-0687-3}{2016(1):177}, 2016.
\newblock

\bibitem{WangGlobalSolvabilityEventual2020}
\textsc{Wang, Y.}:
\newblock {\em Global solvability and eventual smoothness in a chemotaxis-fluid
  system with weak logistic-type degradation}.
\newblock Math. Models Methods Appl. Sci.,
  \href{https://doi.org/10.1142/S0218202520400102}{30(6):1217--1252}, 2020.
\newblock

\bibitem{WangEtAlAsymptoticDynamicsSingular2016}
\textsc{Wang, Z.-A.}, \textsc{Xiang, Z.}, and \textsc{Yu, P.}:
\newblock {\em Asymptotic dynamics on a singular chemotaxis system modeling
  onset of tumor angiogenesis}.
\newblock J. Differ. Equ.,
  \href{https://doi.org/10.1016/j.jde.2015.09.063}{260(3):2225--2258}, 2016.
\newblock

\bibitem{WinklerGlobalLargedataSolutions2012}
\textsc{Winkler, M.}:
\newblock {\em Global large-data solutions in a
  chemotaxis--({{Navier--}}){{Stokes}} system modeling cellular swimming in
  fluid drops}.
\newblock Commun. Partial Differ. Equ.,
  \href{https://doi.org/10.1080/03605302.2011.591865}{37(2):319--351}, 2012.
\newblock

\bibitem{WinklerFinitetimeBlowupHigherdimensional2013}
\textsc{Winkler, M.}:
\newblock {\em Finite-time blow-up in the higher-dimensional
  parabolic\textendash parabolic {{Keller}}\textendash{{Segel}} system}.
\newblock J. Math\'ematiques Pures Appliqu\'ees,
  \href{https://doi.org/10.1016/j.matpur.2013.01.020}{100(5):748--767}, 2013.
\newblock

\bibitem{WinklerStabilizationTwodimensionalChemotaxisNavier2014}
\textsc{Winkler, M.}:
\newblock {\em Stabilization in a two-dimensional
  chemotaxis-{{Navier}}\textendash{{Stokes}} system}.
\newblock Arch. Ration. Mech. Anal.,
  \href{https://doi.org/10.1007/s00205-013-0678-9}{211(2):455--487}, 2014.
\newblock

\bibitem{WinklerLargedataGlobalGeneralized2015}
\textsc{Winkler, M.}:
\newblock {\em Large-data global generalized solutions in a chemotaxis system
  with tensor-valued sensitivities}.
\newblock SIAM J. Math. Anal.,
  \href{https://doi.org/10.1137/140979708}{47(4):3092--3115}, 2015.
\newblock

\bibitem{WinklerTwodimensionalKellerSegel2016}
\textsc{Winkler, M.}:
\newblock {\em The two-dimensional {{Keller}}\textendash{{Segel}} system with
  singular sensitivity and signal absorption: {{Global}} large-data solutions
  and their relaxation properties}.
\newblock Math. Models Methods Appl. Sci.,
  \href{https://doi.org/10.1142/S0218202516500238}{26(05):987--1024}, 2016.
\newblock

\bibitem{WinklerHowFarChemotaxisdriven2017}
\textsc{Winkler, M.}:
\newblock {\em How far do chemotaxis-driven forces influence regularity in the
  {{Navier-Stokes}} system?}
\newblock Trans. Amer. Math. Soc.,
  \href{https://doi.org/10.1090/tran/6733}{369(5):3067--3125}, 2017.
\newblock

\bibitem{WinklerGlobalMasspreservingSolutions2018}
\textsc{Winkler, M.}:
\newblock {\em Global mass-preserving solutions in a two-dimensional
  chemotaxis-{{Stokes}} system with rotational flux components}.
\newblock J. Evol. Equ.,
  \href{https://doi.org/10.1007/s00028-018-0440-8}{18(3):1267--1289}, 2018.
\newblock

\bibitem{WinklerRenormalizedRadialLargedata2018}
\textsc{Winkler, M.}:
\newblock {\em Renormalized radial large-data solutions to the
  higher-dimensional {{Keller}}\textendash{{Segel}} system with singular
  sensitivity and signal absorption}.
\newblock J. Differ. Equ.,
  \href{https://doi.org/10.1016/j.jde.2017.10.029}{264(3):2310--2350}, 2018.
\newblock

\bibitem{WinklerSmallmassSolutionsTwodimensional2020}
\textsc{Winkler, M.}:
\newblock {\em Small-mass solutions in the two-dimensional {{Keller--Segel}}
  system coupled to the {{Navier--Stokes}} equations}.
\newblock SIAM J. Math. Anal.,
  \href{https://doi.org/10.1137/19M1264199}{52(2):2041--2080}, 2020.
\newblock

\end{thebibliography}

\end{document}